\documentclass[
  fontsize=12pt,
  a4paper,
  headings=normal,
  twoside=false,
  leqno,
  parskip=half-,
  abstract=true
]{scrartcl}

\usepackage[english]{babel}
\usepackage[utf8]{inputenc}

\usepackage{hyperref}
\hypersetup{
  pdftitle={Periodicity in Delayed Self-Regulation},
  pdfauthor={Alejandro López-Nieto},
  colorlinks=true,
  linkcolor=blue,
  citecolor=blue,
  filecolor=blue,
  urlcolor=blue
}

\usepackage{amsmath,amssymb,amsthm,mathtools}
\usepackage{cite}
\usepackage{tikz-cd}
\usepackage{xpatch}

\xpatchcmd{\proof}{\itshape}{\normalfont\proofnamefont}{}{}
\newcommand{\proofnamefont}{}

\usepackage[format=plain,labelfont=bf,font=small]{caption}
\newcommand{\email}[1]{\href{mailto:#1}{#1}}

\definecolor{refkey}{rgb}{1,0,0}
\definecolor{labelkey}{rgb}{1,0,0}

\newtheorem{theorem}{Theorem}[section]
\newtheorem{lemma}[theorem]{Lemma}
\newtheorem{corollary}[theorem]{Corollary}
\newtheorem{proposition}[theorem]{Proposition}

\theoremstyle{definition}

\newtheorem{example}[theorem]{Example}

\theoremstyle{remark}
\newtheorem{remark}[theorem]{\textbf{Remark}}

\numberwithin{equation}{section}

\usepackage{cellspace}
\usepackage{diagbox}
\usepackage{pgfplots}
\pgfplotsset{compat=newest}
\usepackage{float}
\usepackage{makecell}
\usepackage{enumerate}
\usepackage{wrapfig}

\begin{document}

\title{
Periodicity in delayed self-regulation is a predator-prey process}

\author{Alejandro López-Nieto\thanks{Department of Mathematics, National Taiwan University, No. 1, Sec. 4, Roosevelt Road, 10617 Taipei, Taiwan; \email{alopez@ntu.edu.tw}.}}
\date{ }

\maketitle
\thispagestyle{empty}
\newpage

\tableofcontents
\newpage
\abstract{
We unify two different periodicity mechanisms: delayed self-regulation and planar predator-prey feedback. We consider scalar delay differential equations $\dot x(t) = rf(x(t), x(t - 1))$ where $f$ is monotone in the delayed component. Due to a Poincar\'{e}--Bendixson theorem for monotone delayed feedback systems, the typical global dynamics present periodic orbits as the delay parameter $r$ increases. In this article, we show that, as we vary the delay, each connected component of periodic orbits is an annulus with global coordinates given by the time and the amplitude of the corresponding periodic solutions. On each annulus, the variables $x(t)$ and $x(t - 1)$ solve an integrable ordinary differential equation that satisfies a predator-prey feedback relation. Moreover, we completely characterize the set of periodic solutions of the delay differential equation in terms of two time maps generated by the underlying predator-prey system. }

\section{Introduction}

The Hutchinson delayed logistic equation
\begin{align} \label{hutchinson}
    \dot{{N}}(t) = r{N}(t)(1 - {N}(t - 1)), \quad r > 0,
\end{align}
models intrinsic oscillations in biodemographics; see \cite{Hut48}. In the Hutchinson equation \eqref{hutchinson}, the population density ${N}(t)$ self-inhibits through competition between younger and older individuals ${N}(t)$ and ${N}(t - 1)$. Typically, for a sufficiently large value of the parameter $r > 0$, the solutions of \eqref{hutchinson} converge to an exponentially attracting periodic solution that oscillates around the saturation density ${N}(t) \equiv 1$. After taking the logarithm, the periodic solutions of \eqref{hutchinson} solve a delay differential equation (abbrv. DDE) of the form
\begin{align}\label{reference-dde}
    \dot x(t) &= r f(x(t),x(t-1)),\quad f \in C^k(\mathbb{R}^2, \mathbb{R}), \, k \geq 2, \quad r\partial_2 f \neq 0.
\end{align}
The delayed self-regulation of the population, encoded as $r\partial_2 f \neq 0$ in \eqref{reference-dde}, is called \emph{monotone delayed feedback}. Here, the parameter $r$ is crucial because the time rescaling $t \mapsto rt$ produces the equivalent equation
\begin{align}\label{reference-dde-resc}
    \dot {\hat{x}}(t) = f(\hat{x}(t), \hat{x}(t - r)).
\end{align}
Thus, $r$ is the \emph{delay} of \eqref{reference-dde-resc}.

Since Hutchinson's original work, the delayed self-regulation model \eqref{reference-dde} has been pointed out as the cause of oscillations in blood cell density \cite{MaGla77}, oceanic temperature \cite{SuSch88}, and gene expression in the segmentation clock \cite{Lew03, YoKo20}. Traditionally, there exist two separate types of delayed feedback: \emph{positive} if $r\partial_2 f > 0$ and \emph{negative} otherwise. Each type models delayed autocatalysis and self-inhibition, respectively. 

Mathematically, the DDE \eqref{reference-dde} is infinite-dimensional as it generates an evolution process on $C$ where $C := (C^0([-1, 0], \mathbb{R}), |\cdot|_{\sup})$ is the \emph{history space}. A \emph{periodic solution} of \eqref{reference-dde} is a nonconstant $C^1$-periodic function $x^\ast: \mathbb{R} \to \mathbb{R}$ that satisfies \eqref{reference-dde} for a delay value $r_\ast \neq 0$. The corresponding \emph{orbit} is the function-valued curve
\begin{align}\label{def-orbit}
    \gamma_{\ast} := \left\{x_t^\ast : t\in \mathbb{R}\right\} \subset C,
\end{align}
where $x_t^\ast$ is defined as $x_t^\ast := x^\ast(t + \vartheta)$ with $\vartheta \in [-1, 0]$. The monotone delayed feedback assumption in the DDE \eqref{reference-dde} is a crucial structural feature due to a Poincar\'{e}--Bendixson theorem \cite{MPSe962}, which ensures that the projection
\begin{align}\label{planar-projection-C}
    \begin{array}{rcl}
        P : C & \longrightarrow & \mathbb{R}^2\\
        \varphi & \longmapsto & (\varphi(0), \varphi(-1)),
    \end{array}
\end{align}
$C^k$-embeds the periodic orbits of \eqref{reference-dde} into $\mathbb{R}^2$. Thus, the projection 
\begin{align}
    P\gamma_\ast = \{(x^\ast(t), x^\ast(t - 1)) : t \in \mathbb{R}\}
\end{align}
is a $C^k$-regular Jordan curve. Moreover, if $x^\dagger(t)$ is another periodic solution of \eqref{reference-dde} at delay $r_\dagger$ and $r_\ast = r_\dagger$, then we have a dichotomy 
\begin{align}\label{def-nesting}
    \text{either} \quad P\gamma_\ast \cap P\gamma_\dagger = \emptyset \quad \text{or} \quad P\gamma_\ast = P\gamma_\dagger.
\end{align}
We say that \eqref{def-nesting} is the \emph{nesting property} of the periodic orbits of \eqref{reference-dde}; see \cite[Lemma 5.7]{MPSe962}. 

The goal of this article is to extend the nesting property \eqref{def-nesting} and drop the constraint that both $x^\ast(t)$ and $x^\dagger(t)$ solve \eqref{reference-dde} at the same delay $r_\ast = r_\dagger$. In showing this \emph{delay-independent nesting}, we will show that the periodic orbits of \eqref{reference-dde} produce invariant two-dimensional manifolds for an extended version of \eqref{reference-dde}. Moreover, the dynamics that the DDE \eqref{reference-dde} induces on the resulting manifold are planar integrable ordinary differential equations (ODEs) whose variables satisfy a predator-prey feedback relation. Thus, we unify two different intrinsic mechanisms for periodicity in modeling. That is, periodicity in the infinite-dimensional, delayed self-regulation DDE \eqref{reference-dde} is a two-dimensional, predator-prey ODE process if we consider $x(t)$ and $x(t - 1)$ as separate species.

Our motivation is based on the numerical approximation in Figure \ref{fig1}, where we plot the projections of the periodic orbits $P\gamma_r$ of the Hutchinson equation \eqref{hutchinson} at different delay values $r$. The resulting Jordan curves do not intersect. Moreover, in \cite{lop24, Lo23}, we have discussed the monotone delayed feedback DDE \eqref{reference-dde} under the additional \emph{symmetric feedback} assumptions
\begin{align}\label{def-sym-feedback}
    f(-u, v) = f(u, v) \quad \text{and} \quad f(u, -v) = - f(u, v).
\end{align}
Then, up to a nondegeneracy condition \cite[Theorem 1.3]{lop24}, all periodic solutions $x^\ast(t)$ of \eqref{reference-dde} yield periodic solutions $(x^\ast(t), x^\ast(t - 1))$ of the planar ODEs
\begin{align}\label{even-odd-ODEs}
    \left(
    \begin{array}{c}
        \dot u \\
        \dot v
    \end{array}
    \right)
    =
    r\left(
    \begin{array}{c}
        f(u, v) \\
        - f(v, u)
    \end{array}
    \right), \quad (u, v) \in \mathbb{R}^2
\end{align}
for a suitable value of the delay $r$. Since $r$ is a constant time scaling in the ODEs \eqref{even-odd-ODEs}, the nesting property \eqref{def-nesting} holds, even if $x^\ast(t)$ and $x^\dagger(t)$ solve the DDE \eqref{reference-dde} at different values $r_\ast \neq r_\dagger$. A key aspect of the symmetric feedback \eqref{def-sym-feedback} is that it enforces a rational minimal period $p_\ast$ on all the periodic solutions of \eqref{reference-dde}. If the period $p_\ast$ is rational, then there exists an $M\in \mathbb{N}$ such that $x^\ast(t - M) = x^\ast(t)$. Hence, the $M$-vector $u^j(t) := x^\ast(t - j)$ solves the ODEs
\begin{align}\label{def-cycl-feed}
    \dot u^j(t) = r f(u^j(t), u^{j + 1}(t)),\quad j \mod M.
\end{align}
Together with the monotone delayed feedback assumption $r\partial_2f \neq 0$, the ODEs \eqref{def-cycl-feed} are a \emph{monotone cyclic feedback} system and possess a nesting property much like \eqref{def-nesting}; see \cite{MPSm90}. Like in \eqref{even-odd-ODEs}, the delay $r$ is a time scaling in \eqref{def-cycl-feed} and the nesting property \eqref{def-nesting} holds for any two periodic solutions $x^\ast(t)$ and $x^\dagger(t)$ of the DDE \eqref{reference-dde} that possess rational periods, regardless of the delays $r_\ast$ and $r_\dagger$. In contrast, if $p_\ast$ is irrational, then the monotone cyclic feedback ODEs \eqref{def-cycl-feed} are defined in $\mathbb{R}^\mathbb{Z}$ and the results in \cite{MPSm90} do not apply.

This article is organized as follows. Section \ref{Sec2} presents the main results and ideas in Theorems \ref{thm1}--\ref{thm4}. In Section \ref{Sec3}, we show that the periodic orbits of \eqref{reference-dde} always admit a local continuation in a real parameter $b$. Section \ref{Sec4} shows that the parameter $b$ is locally equivalent to the amplitude of the periodic solutions, allowing us to globalize the continuation and prove Theorem \ref{thm1}, Theorem \ref{thm2}, and Theorem \ref{thm3}. Section \ref{Sec5} contains auxiliary lemmata used in Section \ref{Sec4} that also allow us to prove Theorem \ref{thm4} in Section \ref{Sec6}. Finally, Section \ref{Sec7} proves auxiliary results used in the local continuation of Section \ref{Sec3}.

\section{Main results}\label{Sec2}

Let $x^\ast(t)$ solve the monotone delayed feedback DDE \eqref{reference-dde} at delay value $r_\ast \neq 0$. We include the delay into the phase space by considering the \emph{extended DDE}
\begin{align}\label{extended-DDE}
    \begin{split}
        \dot x(t) &= r f(x(t),x(t-1)), \\
        \dot r(t) &= 0,
    \end{split} \quad f \in C^k(\mathbb{R}^2, \mathbb{R}), \, k \geq 2, \quad r\partial_2 f \neq 0.
\end{align}
Trivially, the periodic orbit $\gamma_\ast$ of \eqref{reference-dde} yields an orbit $(\gamma_\ast; r_\ast)$ of the extended DDE \eqref{extended-DDE} in the extended phase space $C \times \mathbb{R}$. Naturally, $(\gamma_\ast; r_\ast)$ is contained in the periodic set
\begin{align}\label{definition-set-periodic-orbits}
    \mathcal{P} := \left\{(x_t^\ast; r_\ast) : t\in \mathbb{R} \text{ and } (x^\ast(t);r_\ast) \text{ is a periodic solution of \eqref{extended-DDE}}\right\} \subset C \times \mathbb{R}.
\end{align}
We emphasize that all periodic points in $\mathcal{P}$ are nontrivial, thus, we exclude the equilibria of \eqref{extended-DDE} from $\mathcal{P}$. In analogy to the periodic orbits \eqref{def-orbit}, we consider connected components, alias \emph{periodic branches}, of $\mathcal{P}$. Moreover, we formalize the idea of ``ignoring the delay $r$'' in \eqref{def-nesting} by defining the extended projection
\begin{align}\label{ext-proj}
    \begin{array}{rcl}
        \bar{P}: C \times \mathbb{R} & \longrightarrow & \mathbb{R}^2 \\
        (\varphi; r) & \longmapsto & P\varphi.
    \end{array}
\end{align}

\begin{theorem}[Branch projection]\label{thm1}
    Let $\mathcal{B}$ be a periodic branch of the extended DDE \eqref{extended-DDE} and denote $\mathcal{O} = \bar{P}\mathcal{B}$. Then the projection $\bar{P} : \mathcal{B} \to \mathcal{O}$ is a $C^k$-diffeomorphism. 
\end{theorem}

That is, analogously to the Poincar\'{e}--Bendixson theorem \cite{MPSe962} for the projection \eqref{planar-projection-C}, the extended projection \eqref{ext-proj} $C^k$-embeds periodic branches of the extended DDE \eqref{extended-DDE} into $\mathbb{R}^2$. In particular, Theorem \ref{thm1} shows that if $(\gamma^1;r_1)$ and $(\gamma^2; r_2)$ are two distinct periodic orbits of the extended DDE \eqref{extended-DDE} that lie on the same branch $\mathcal{B}$, then the planar projections $P\gamma_1$ and $P\gamma_2$ are nested within one another, regardless of the delays $r_\ast$ and $r_\dagger$.
Theorem \ref{thm1} has an analogue in scalar reaction-diffusion partial differential equations (PDEs) on a circular domain. In \cite{FiRoWo04}, thanks to a Poincar\'{e}--Bendixson theorem, the PDE rotating waves are embedded in two dimensions via a projection that ignores the wave speed. Inspired by the PDE scenario, we say that $\mathcal{O}\subset \mathbb{R}^2$ in Theorem \ref{thm1} is the \emph{cyclicity component} of $\mathcal{B}$.

A major drawback of Theorem \ref{thm1} is that it regards orbits as sets, ignoring any dynamics on the cyclicity component $\mathcal{O}$. To restore the dynamics, we first define the \emph{amplitude} of a periodic solution $(x^\ast(t); r_\ast)$ of \eqref{extended-DDE} as the maximum of $x^\ast(t)$. In analogy, the \emph{amplitude domain} of a branch $\mathcal{B}$ is the set of amplitudes of all periodic solutions with points in $\mathcal{B}$. That is, the interval $(\underline{a}, \overline{a})$ with the bounds defined as
\begin{align}\label{def-amp-dom}
     \underline{a}:= \inf \left\{\max_t x(t) : (x_0; r) \in \mathcal{B} \right\} \quad \text{and} \quad \overline{a}:= \sup \left\{\max_t x(t) : (x_0; r) \in \mathcal{B} \right\}.
\end{align}

\begin{theorem}[Time-amplitude parametrization]\label{thm2}
    Let $\mathcal{B}$ be a periodic branch of the extended DDE \eqref{extended-DDE}. If we denote the amplitude domain \eqref{def-amp-dom} of $\mathcal{B}$ by $(\underline{a}, \overline{a})$, then there exist $C^k$-functions $p : (\underline{a}, \overline{a}) \to (0, \infty)$ and $r : (\underline{a}, \overline{a}) \to \mathbb{R}$, and a $C^k$-family $(x^a(t); r(a))$ of periodic solutions of \eqref{extended-DDE} such that:
    \begin{enumerate}
        \item For all $a\in (\underline{a}, \overline{a})$, $p(a)$ is the minimal period of $x^a(t)$ and
        \begin{align}
            \max_t x^a(t) = x^a(0) = a.
        \end{align}
        \item The solutions obtained in this way parametrize $\mathcal{B}$, that is,
        \begin{align}
            \mathcal{B} = \left\{(x_t^a; r(a)) : t\in \mathbb{R}, \; a\in (\underline{a}, \overline{a})\right\}.
        \end{align}
    \end{enumerate}
    Moreover, if we define $\mathbb{A}:=\{(t, a) : t \in \mathbb{R}/p(a)\mathbb{Z},\; a\in (\underline{a}, \overline{a})\}$, then the map 
    \begin{align}
        \begin{array}{rrcl}
            \beta:& \mathbb{A} & \longrightarrow & \mathcal{B} \\
            &(t, a) & \longmapsto & (x_{t}^a; r(a))
        \end{array}
    \end{align}
    is a $C^k$-diffeomorphism.
\end{theorem}

In other words, Theorem \ref{thm2} shows that the branches of periodic points $\mathcal{B}$ of the extended DDE \eqref{extended-DDE} are annuli $\mathbb{A}$ with a single global chart $\beta$. Since the global branch parameters are time and amplitude, we say that the representation $\beta$ of the branch $\mathcal{B}$ in Theorem \ref{thm2} is the \emph{time-amplitude parametrization} of $\mathcal{B}$. In particular, the branches of periodic points do not have turns in amplitude. Each amplitude labels a unique periodic orbit within each branch and prevents the formation of isolas of periodic orbits in the DDE \eqref{reference-dde}. Combining the branch projection in Theorem \ref{thm1} with the time-amplitude parametrization in Theorem \ref{thm2} results in the following.

\begin{theorem}[Predator-prey reduction]\label{thm3}
    Under the assumptions of Theorem \ref{thm2}, consider the time-amplitude parametrization $(x_t^a; r(a))$ of $\mathcal{B}$. Then there exists a $C^k$-function $\alpha: \mathcal{O} \to \mathbb{R}$ that satisfies 
    \begin{align}\label{def-alpha}
        \alpha\left(\bar{P}(x_t^a; r(a))\right) = a \; \quad \text{for all } (x_t^a; r(a)) \in \mathcal{B}.
    \end{align}
    Moreover, there exists a $C^k$-function $g: \mathcal{O} \to \mathbb{R}$ such that 
    \begin{align}\label{pr-pr-feedback}
        \partial_1 g(u, v) \partial_2 f(u, v) < 0  \quad \text{for all } (u, v)\in \mathcal{O},
    \end{align}
    and, for all $a \in (\underline{a}, \overline{a})$, $\bar{P}(x_t^a; r(a))$ solve the planar ODEs
    \begin{align}\label{planar-ODE}
        \left(\begin{array}{c}
            \dot u \\
            \dot v
        \end{array}\right) = r(\alpha(u,v))
        \left(\begin{array}{c}
            f(u,v) \\
            g(u,v)
        \end{array}\right), \quad (u, v) \in \mathcal{O},
    \end{align}
\end{theorem}
The predator-prey reduction in Theorem \ref{thm3} shows that the periodic orbits of the extended DDE \eqref{extended-DDE} foliate the cyclicity component $\mathcal{O}$ by level sets of a differentiable first integral $\alpha$. By the identity \eqref{def-alpha}, $\alpha$ extracts the amplitude of a periodic point $(x_t^a, r(a))$ using only the two-point evaluation $(x^a(t), x^a(t - 1))$. In particular, we recover the amplitude domain of $\mathcal{B}$ as the interval $\alpha(\mathcal{O})$. Moreover, the ODEs \eqref{planar-ODE} satisfy a predator-prey feedback relation \eqref{pr-pr-feedback}.

Next, we discuss the relative position of the periodic branches in the extended phase space $C \times \mathbb{R}$. We highlight that there exist different branches $\mathcal{B}$ and $\hat{\mathcal{B}}$ of the extended DDE \eqref{extended-DDE} with overlapping cyclicity components. Indeed, if $p_\ast$ is a period of $x^\ast(t)$, then, by substitution,
\begin{align}\label{def-resc-sym}
    x^\ast((1 + m p_\ast) t) \text{ is a solution of \eqref{reference-dde} at delay }(1 + mp_\ast) r_\ast \text{ for all }m \in \mathbb{Z}.
\end{align}
Moreover, since the planar projections of the copies satisfy
\begin{align}
    \left\{(x^\ast(t), x^\ast(t - 1)) : t \in \mathbb{R}\right\} = \left\{(x^\ast((1 + m p_\ast) t), x^\ast((1 + m p_\ast)(t - 1))) : t \in \mathbb{R}\right\},
\end{align}
the \emph{time rescaling symmetry} \eqref{def-resc-sym} produces infinitely many periodic branches of \eqref{extended-DDE} with identical cyclicity components. 
\begin{example}\label{example1}
    In the Hutchinson equation \eqref{hutchinson}, all periodic solutions appear by successive Hopf bifurcations from the constant solution ${N}(t) \equiv 1$ as the size of the delay $|r|$ increases. Thus, the inside boundary of all periodic branches consists of a Hopf point and the cyclicity component is the single annulus $\mathcal{O} = (0, \infty)^2 \setminus \{(1, 1)\}$ whose inner radius is zero. The explicit form of the predator-prey reduction \eqref{planar-ODE} is unknown, but we provide a numerical approximation of the integral curves with amplitudes smaller than five in Figure \ref{fig1}.
\end{example}
\begin{figure}[h]
    \centering
    \includegraphics[width = 0.87\textwidth]{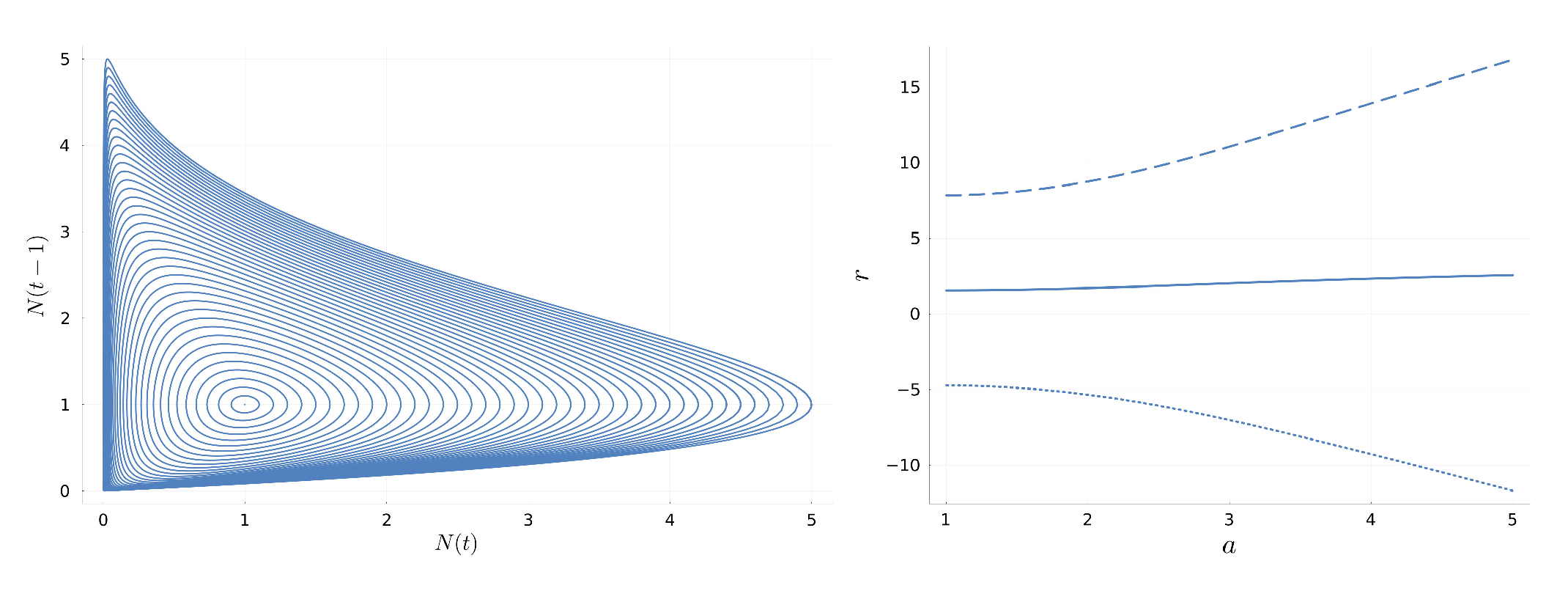}
    \caption{ \em Left:  Numerical approximation of the integral curves of the planar ODEs \eqref{planar-ODE} for the Hutchinson DDE \eqref{hutchinson}. The cyclicity component $\mathcal{O}$ consists of one annulus $(0,\infty)^2 \setminus \{(1, 1)\}$. Right: The three branches are the graphs of the delay map $r(a)$ (solid) and the time rescaling symmetries \eqref{def-resc-sym} given by $(1 + p(a))r(a)$ (dashed), and $(1 - p(a))r(a)$ (dotted). All periodic branches connect to a Hopf bifurcation at amplitude one and have the same cyclicity component. The amplitude domain $\alpha(\mathcal{O})$ of all the branches is $(1, \infty)$.}
    \label{fig1}
\end{figure}
\begin{theorem}[Delay-independent nesting]\label{thm4}
    Let $\mathcal{B}$ and $\hat{\mathcal{B}}$ be two branches of periodic points of the extended DDE \eqref{extended-DDE} and consider the corresponding maps $p, \, \hat{p}, \, r,$ and $\hat{r}$ from the time-amplitude parametrization in Theorem \ref{thm2}. If $\bar{P}\mathcal{B} \cap \bar{P}\hat{\mathcal{B}} \neq \emptyset$, then $\bar{P}\mathcal{B} = \bar{P}\hat{\mathcal{B}}$,
    \begin{align}
        \hat{r} = (1 + mp) r, \quad \text{and} \quad  \hat{p} = \frac{p}{1 + mp}, \quad \text{for some }m\in\mathbb{Z}.
    \end{align}
\end{theorem}

In particular, Theorem \ref{thm4} shows that the cyclicity components $\mathcal{O}$ of the branches $\mathcal{B}$ are disjoint, except for branches related by the time rescaling symmetry \eqref{def-resc-sym}. This motivates the definition of the \emph{cyclicity set}
\begin{align}
    \mathcal{C} := \bar{P} \mathcal{P}.
\end{align}
Naturally, $\mathcal{C}$ is the union the annular cyclicity components $\mathcal{O}$ of the branches $\mathcal{B}$. In Example \ref{example1}, the cyclicity set consists of a single component. Theorem \ref{thm4} allows us to extend the functions $\alpha$ and $g$ in Theorem \ref{thm3} to the whole cyclicity set. Thus, we obtain a complete description of the  periodic solutions of \eqref{reference-dde} via the following corollary.

\begin{corollary}\label{cor5}
    Let $\mathcal{C}$ be the cyclicity set of the extended DDE \eqref{extended-DDE}. There exist $C^k$-functions $\alpha, g: \mathcal{C} \to \mathbb{R}$ such that $\partial_1 g(u,v) \partial_2 f(u, v) < 0$ for all $(u, v) \in \mathcal{C}$ and, if $x(t)$ is a periodic solution of \eqref{reference-dde} with amplitude $a$ at delay $r$, then $\alpha(x(t), x(t -1)) = a$ for all $t \in \mathbb{R}$ and the vector $(x(t), x(t - 1)) \in \mathbb{R}^2$ satisfies
    \begin{align}
        \begin{split}
            \dot x(t) = rf(x(t), x(t - 1)),\\
            \dot x(t - 1) = r g(x(t), x(t - 1)),
        \end{split} \quad \text{for all }t\in \mathbb{R}.
    \end{align}
    Conversely, $\alpha$ is a first integral of the planar ODEs
    \begin{align}\label{def-planar-odes}
        \left(\begin{array}{c}
            \dot u\\
            \dot v
        \end{array}\right) = \left(\begin{array}{c}
            f(u, v) \\
            g(u, v)
        \end{array}\right), \quad (u, v) \in \mathcal{C}.
    \end{align}
    Moreover, there exist $C^k$-functions $p, r: \alpha(\mathcal{C}) \to (0, \infty)$ such that all solutions $(u(t), v(t))$ of \eqref{def-planar-odes} have minimal period 
    \begin{align}
        p(\alpha(u(t),v(t)))r(\alpha(u(t),v(t)))
    \end{align}
    and satisfy
    \begin{align}
        v(t) = u(t - r(\alpha(u(t),v(t)))), \quad \text{for all }t\in \mathbb{R}.
    \end{align}
\end{corollary}
\begin{proof}
    By Theorem \ref{thm4}, we can pick one periodic branch associated to each connected component of $\mathcal{C}$ and use Theorem \ref{thm2} and Theorem \ref{thm3} to define the maps $\alpha$, $g$, $r$, and $p$ on each cyclicity component.  
\end{proof}
\begin{remark}\label{remark1}
    By the rescaling symmetry \eqref{def-resc-sym}, the \emph{delay map} $r(a)$ and the \emph{period map} $p(a)$ in Corollary \ref{cor5} are nonunique. However, their product is unique by uniqueness of the minimal period of \eqref{def-planar-odes}
\end{remark}

\begin{example}\label{example2}
Recall the symmetric feedback example \eqref{def-sym-feedback}. Following \cite{lop24}, the cyclicity set is 
\begin{align}
    \mathcal{C} = \mathbb{R}^2 \setminus \{(0, 0)\},
\end{align}
and, explicitly, $g(u, v) = -f(v, u)$. In general, we do not have closed formulas for the delay map $r(a)$ and period map $p(a)$ in Corollary \ref{cor5}, but the symmetric feedback \eqref{def-sym-feedback} implies a constant ratio
\begin{align}
    \frac{p(a)}{r(a)} \in \mathbb{Q}, \quad \text{for all }a\in (0, \infty).
\end{align}
Consider the special case of the enharmonic oscillator \cite{Lo23}, that is, the DDE \eqref{reference-dde} of the form
\begin{align}\label{enharmODE}
    \dot x(t) = -r\Omega\left(x(t)^2 + x(t - 1)^2\right)x(t - 1),
\end{align}
where $\Omega: \mathbb{R} \to \mathbb{R}$ is a positive frequency function. Then, one possible choice for the maps in Corollary \ref{cor5} is
\begin{align}
    \alpha(u, v) = \sqrt{u^2 + v^2}, \quad p(a) = \frac{2\pi}{\Omega(a^2)}, \quad r(a) = \frac{\pi}{2\Omega(a^2)}.
\end{align}

\end{example}

\begin{example}\label{example3}
Beyond the symmetric feedback \eqref{def-sym-feedback}, consider the DDE \eqref{reference-dde} with the nonlinearity
\begin{align}\label{QRT-nonlin}
    f(u, v) =  - \partial_v H(u, v),
\end{align}
where $H$ is the biquadric function
\begin{align}\label{def-Hamiltonian}
    H(u, v) = u^2 v^2 +(u^2 v + u v^ 2) + \frac{1}{2} (u^2 + v^2).
\end{align}
The Hamiltonian \eqref{def-Hamiltonian} produces the ODEs
\begin{align}\label{def-Hamode}
    \left(
    \begin{array}{c}
        \dot u \\
        \dot v
    \end{array}
    \right)
    =
    r\left(
    \begin{array}{c}
        - \partial_v H(u, v) \\
        \partial_u H(u, v)
    \end{array}
    \right), \quad (u, v) \in \mathbb{R}^2,
\end{align}
where the phase portrait is of \eqref{def-Hamode} is of double-well or Duffing-type. That is, it consists of periodic solutions except for the level set $H = 0$, consisting of two equilibria at $(-1, -1)$ and $(0, 0)$ and a double homoclinic figure-eight at $H = 1/16$; see Figure \ref{fig2}. The double-well \eqref{def-Hamiltonian} is an instance of the five-parameter family
\begin{align}\label{def-QRT-Ham}
    H_1 u^2 v^2 + H_2 (u^2 v  + u v^2) + H_3 (u^2 + v^2) + H_4 u v + H_5 (u + v), \quad H_j \in \mathbb{R}.
\end{align}
The Hamiltonians \eqref{def-QRT-Ham} are crucial to the so-called QRT maps, special rational transformations of $\mathbb{R}^2$ that preserve the level sets of \eqref{def-QRT-Ham}; see \cite{Duis10}. The key features of \eqref{def-QRT-Ham} are that $H$ is invariant under the reflection $(u, v) \mapsto (v, u)$ and the Galois switch $(u^+, v) \mapsto (u^-, v)$ that exchanges the two different values $u^+$ and $u^-$ such that
\begin{align}
    H(u^+, v) = H(u^-, v).
\end{align}
Explicitly, the double-well Hamiltonian ODEs \eqref{def-Hamode} are equivariant under the QRT map 
\begin{align}\label{def-exp-QRT-map}
    (u, v) \mapsto \left(v, -u - \frac{2v^2}{2v^2 + 2v + 1}\right).
\end{align}
Moreover, all orbits of \eqref{def-Hamode} are invariant, as sets, under \eqref{def-exp-QRT-map}. Hence, \eqref{def-exp-QRT-map} is a time-map for \eqref{def-Hamode} and some of the elements in Corollary \ref{cor5} are
\begin{align}
    \mathcal{C} = \mathbb{R}^2 \setminus \left(H^{-1}\left(\frac{1}{16}\right) \cup H^{-1}(0)\right) \quad \text{and} \quad g(u, v) = \partial_1 H(u, v).
\end{align}
Notice that the Hamiltonian $H$ is preserved by \eqref{def-Hamode} and is locally equivalent to $\alpha$. However, none of $\alpha$, $r(a)$, or $p(a)$ is known explicitly; see Figure \ref{fig2} for a numerical approximation.
\end{example}

\begin{figure}[t]
    \centering
    \includegraphics[width = 0.87\textwidth]{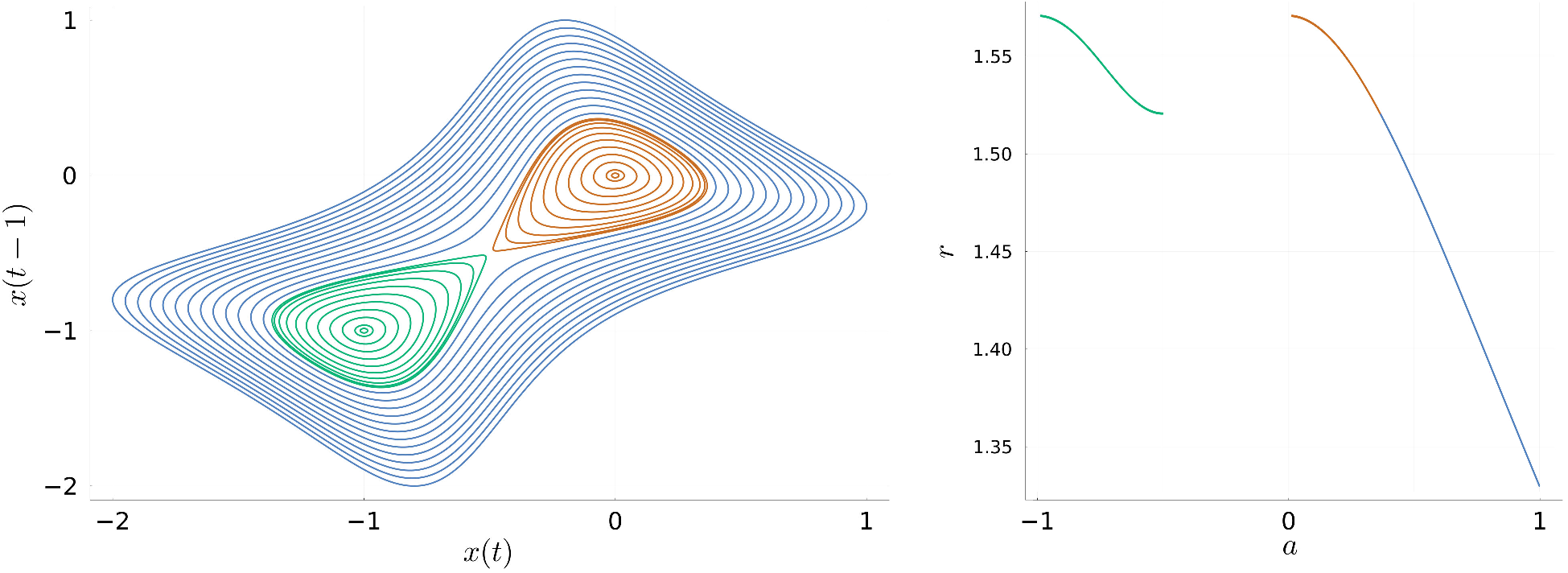}
    \caption{\em Left: Numerical approximation of the integral curves of three branches of periodic solutions of the DDE \eqref{reference-dde} with the QRT nonlinearity \eqref{QRT-nonlin}. The cyclicity set $\mathcal{C}$ consists of three connected components in blue, green, and orange separated by a homoclinic figure-eight. Right: The delay map $r(a)$ for the three branches shows that the periodic solutions appear by Hopf bifurcation at amplitudes $-1$ and $0$, and by a homoclinic bifurcation from the figure-eight at amplitudes $-1/2$ and $(\sqrt{3} - 1)/2$. Hence the delay map $r(a)$ and period map $p(a)$ have the domain $(-1 , -1/2) \cup (0, (\sqrt{3} - 1)/2) \cup ((\sqrt{3} - 1)/2, \infty)$. Since the period map approaches infinity at the figure-eight, the time rescaled branches $(1 + mp(a))r(a)$, with $m \neq 0$, become unbounded at the amplitudes $-1/2$ and $(\sqrt{3} - 1)/2$.}
    \label{fig2}
\end{figure}
\paragraph{Conclusion and discussion}
In Theorem \ref{thm1}, Theorem \ref{thm2}, and Theorem \ref{thm3}, we have shown that the periodic branches of the extended DDE \eqref{extended-DDE} are annuli diffeomorphic to their cyclicity component via the extended projection \eqref{ext-proj}. Each amplitude corresponds to a unique periodic orbit on each branch and the shape of the branch is given by the delay map $r(a)$ from the time-amplitude parametrization in Theorem \ref{thm2}. Moreover, the components $x(t)$ and $x(t - 1)$ satisfy integrable ODEs with predator-prey feedback. For example, in the Hutchinson equation \eqref{hutchinson}, the younger and older individuals $N(t)$ and $N(t - 1)$ behave like distinct species where $N(t)$ is the prey and $N(t - 1)$ is the predator. 

In Theorem \ref{thm4}, we have shown that the cyclicity components of the periodic branches are disjoint except for branch copies appearing by the time rescaling symmetry \eqref{def-resc-sym}. Thus, each cyclicity component is associated to a unique family of periodic branches generated by the delay map $r(a)$ and the period map $p(a)$ obtained in Theorem \ref{thm2}. In particular, this allows us to combine our results into Corollary \ref{cor5}. Thus we show that there exist integrable predator-prey ODEs with domain a planar cyclicity set $\mathcal{C}$ that generate all periodic solutions of the DDE \eqref{reference-dde}. We highlight six further consequences of our results.

\begin{enumerate}
    \item Our novel tools completely explain the bifurcation structure of periodic solutions in the scalar DDE \eqref{reference-dde} as the delay $r$ changes. From our viewpoint, new periodic solutions can appear in \eqref{reference-dde} in two ways: in the interior of a periodic branch or at the boundary. On the one hand, following \cite{MPNu13}, the only candidates in the interior are saddle-node bifurcations where a nonhyperbolic periodic orbits splits in two. This is the situation at nondegenerate critical values of the delay map $r(a)$ in Corollary \ref{cor5}. On the other hand, the effects at the boundary of a periodic branch are richer and can be local, like the Hopf bifurcation in Figure \ref{fig1}, as well as global, like the homoclinic bifurcation in Figure \ref{fig2}. 

    \item We highlight that at most one of the branch copies obtained by the time rescaling symmetry \eqref{def-resc-sym} possesses homoclinic orbits at the boundary. We call \emph{slow branch} to the unique choice whose period map $p(a)$ in Remark \ref{remark1} satisfies $p > 2$. The slow branch contains the \emph{slowly oscillating periodic solutions} (SOPS), that is, the periodic solutions of the DDE \eqref{reference-dde} whose extrema are separated by at least one time unit; see also Lemma \ref{lemma-relative-zero-position} in Section \ref{Sec7}. In particular, SOPS are the only periodic solutions of \eqref{reference-dde} whose minimal periods can become unbounded. Thus, only SOPS can accumulate to homoclinic orbits; see Example \ref{example3}. Furthermore, due to an eigenvalue structure \cite{MPNu13}, SOPS are the only periodic solutions of \eqref{reference-dde} that can be exponentially attracting. 

    \item If we rescale time in the slow branch via \eqref{def-resc-sym}, the delay map $(1 + mp(a))r(a)$ of the remaining branch copies becomes unbounded as the minimal period $p(a)$ grows to infinity at amplitudes corresponding to homoclinic orbits. This can be seen in Example \ref{example3}: as the delay $|r|$ grows to infinity, the QRT equation \eqref{QRT-nonlin} possesses periodic solutions that spend increasingly large amounts of time close to an equilibrium and quickly spike, reaching the homoclinic amplitude before relaxing back to equilibrium. Such temporally localized states are sometimes called \emph{temporal dissipative solitons}; see \cite{Yanchuk2019}. 
    
    \item The unstable dimension of the periodic orbits correlates to the sign of the derivative $r'(a)$. As we mentioned above, the local stability on a periodic branch only changes at saddle-node bifurcations that correspond to nondegenerate critical values of the delay map $r(a)$. Thus, the spectral structure in \cite{MPNu13} ensures that the sign changes of $r'(a)$ indicate if the unstable dimension of the periodic orbits increases or decreases by one, up to the determination of an orientation. In \cite{lop24}, we have computed the unstable dimension of the periodic orbits in the enharmonic oscillator of Example \ref{example2}, relative to $r(a)$. However, our proof relies on a reduced characteristic equation that follows from the symmetric feedback assumption \eqref{def-sym-feedback}.
    \item Theorem \ref{thm4} removes the need for a nondegeneracy assumption that we required in \cite{lop24, Lo23}. In particular, we have shown that all the periodic solutions $x^\ast(t)$ of the DDE \eqref{reference-dde} with symmetric feedback \eqref{def-sym-feedback} satisfy the odd symmetry
    \begin{align}
        x^\ast(t - 2) = - x^\ast(t).
    \end{align}
    \item We have characterized the time-periodic traveling waves in lattice differential equations of the form 
    \begin{align}\label{def-lat}
        \dot u^j(t) = r f(u^j(t), u^j(t)), \quad j \in \mathbb{Z}.
    \end{align}
    By Theorem \ref{thm1}, the time-periodic traveling waves of \eqref{def-lat} sit on a two-dimensional manifold. Moreover, if we impose the space-periodicity $u^j(t) \equiv u^{j + M}(t)$ for $M\in \mathbb{N}$, then the time-periodic traveling waves of the lattice differential equation \eqref{def-lat} yield traveling waves of the monotone cyclic feedback ODEs \eqref{def-cycl-feed}.
\end{enumerate}

\paragraph{Further investigations}
\begin{enumerate}
    \item[(i)] Our results show that the richest source of complexity in the DDE \eqref{reference-dde} are the boundaries of the cyclicity set $\mathcal{C}$. So far, we have not addressed which regions of $\mathbb{R}^2$ are realizable as cyclicity sets of \eqref{reference-dde}. A first attempt was made in \cite{vas17} by describing the nesting combinatorics of periodic solutions in terms of parenthetical expressions. However, a finer description is required to understand the possible global bifurcation phenomena in \eqref{reference-dde} precisely. 

    \item[(ii)] Having a better understanding of the delay map $r(a)$ is crucial because it determines the shape and stability of the branches of periodic orbits. We emphasize that having an explicit formula for the delay and period maps is generally unfeasible, even in ODEs. However, in the enharmonic oscillator \eqref{enharmODE} of Example \ref{example2}, the relative position, in terms of amplitudes and cyclicity components, of the periodic solutions of \eqref{enharmODE} provides sufficient information to reconstruct the connection graph of \eqref{reference-dde}; see \cite{Lo23}. We expect to obtain an analogous characterization to the \emph{period map signature} for reaction-diffusion PDEs developed in \cite{rocha07}.
    
    \item[(iii)] We hint at a deeper link between the monotone delayed feedback DDE \eqref{reference-dde} and reaction-diffusion PDEs on the circle. Notice that, as a by-product of the predator-prey reduction in Theorem \ref{thm3}, we can recast the planar ODEs \eqref{planar-ODE} into the second-order ODE
    \begin{align}\label{sec-ord}
        \ddot u = \tilde{f}(u, \dot u) := \partial_1 f(u, v) \dot u + \partial_2 f(u, v) g(u, v).
    \end{align}
    The periodic orbits of \eqref{sec-ord} are the rotating wave solutions of translation-equivariant reaction-diffusion PDEs on the circle. Following \cite{FiRoWo04}, the period map of \eqref{sec-ord} encodes the connection graph of the original PDE. We expect that there exists a DDE equivalent, but we lack a common framework enveloping both DDE and PDE settings. 
\end{enumerate}

\section{Local continuation}\label{Sec3}
We regard the DDE \eqref{reference-dde} as the generator of an evolution process on the Banach space $C$ consisting of $C^0([-1,0], \mathbb{R})$ equipped with the supremum norm \cite{HaLu93}. The DDE \eqref{reference-dde} possesses a solution \emph{semiflow}, that is, there exists a $C^k$-map $S: \mathbb{R} \times C \times \mathbb{R} \to C$ such that for any solution $x(t)$, $t\geq -1$, of \eqref{reference-dde} at delay $r$, we have
\begin{align}
    S(t, x_0; r)(\vartheta) := x_t(\vartheta) = x(t + \vartheta).
\end{align}
The value $x_0$ is the \emph{initial condition} of the \emph{orbit}
\begin{align}
    \gamma := \{x_t : t \geq 0\}.
\end{align}
In the remainder of the section, $\gamma_{\ast}$ denotes the orbit of a periodic solution ${x}^\ast(t)$ of \eqref{reference-dde} at delay $r_\ast$ and  ${p_\ast} > 0$ denotes the minimal period. A crucial property of ${x}^\ast(t)$ is that it is a \emph{simple oscillation}. More precisely, ${x}^\ast(t)$ attains its maximum and minimum once over a minimal period, and is monotone otherwise; see \cite[Theorem 7.1]{MPSe962}. In contrast to the amplitude of ${x}^\ast(t)$, the minimum $\min_{t} {x}^\ast(t)$ is called the \emph{depth}. Without loss of generality, we assume that ${x}^\ast(t)$ is a \emph{normalized periodic solution}, that is, that ${x}^\ast(0)$ is a maximum. The first time $q \in (0, {p_\ast})$ such that ${x}^\ast(q)$ is a minimum is called \emph{time of depth} of ${x}^\ast(t)$.

\begin{lemma}\label{lemma-basic-periodic}
    Let ${x}^\ast(t)$ be a normalized periodic solution of the DDE \eqref{reference-dde} at delay $r_\ast \neq 0$ with minimal period ${p_\ast}$. Then, there exists an $n \in \mathbb{N}$ such ${p_\ast} \in {J}_n$ where
    \begin{align}\label{In-definition}
            \begin{split}
                {J}_n:= \left(\frac{2}{n}, \frac{2}{n - 1}\right).
            \end{split}
    \end{align}
    Moreover, if $q \in (0, {p_\ast})$ is the time of depth of ${x}^\ast(t)$, then we have that
    \begin{align}\label{product-property}
        \ddot{x}^\ast (0) \ddot{x}^\ast(q) < 0, \quad \dot {x}^\ast(1)\dot {x}^\ast(-1) < 0, \quad \text{and} \quad \dot {x}^\ast(q + 1) \dot {x}^\ast(q - 1) < 0.
    \end{align}
\end{lemma}
\begin{proof}
    First, it is well known that the minimal period belongs to one of the intervals $J_n$ in \eqref{In-definition} because neither $1$ nor $2$ can be a period of ${x}^\ast(t)$. The reason is that neither
    \begin{align}
        \dot u(t) = r f(u(t), u(t)) \quad \text{nor} \quad \left(\begin{array}{c}
            \dot u^1(t) \\
            \dot u^2(t)
        \end{array}\right) = r\left(\begin{array}{c}
            f(u^1(t), u^2(t))  \\
            f(u^2(t), u^1(t))
        \end{array}\right),
    \end{align}
    possess periodic solutions; see \cite{ChMP78}.
    Second, we show the inequalities \eqref{product-property}. By \cite[Theorem 5.1]{MPNu13}, all zeros of $\dot {x}^\ast(t)$ are simple. Hence, 
    \begin{align}
        \begin{split}
            \ddot {x}^\ast(0) &= r_\ast\partial_2 f({x}^\ast(0), {x}^\ast(-1)) \dot {x}^\ast(-1) < 0 \quad \text{and} \\
            \ddot {x}^\ast(q) &= r_\ast\partial_2 f({x}^\ast(q), {x}^\ast(q - 1)) \dot {x}^\ast(q - 1) > 0.
        \end{split}
    \end{align}
    Next, we show that $\dot {x}^\ast(1) \dot {x}^\ast(- 1) < 0$, the situation at the time of depth $q$ is analogous. By \cite[Theorem 2.1]{MPSe962}, the set 
    \begin{align}
        P\gamma_{\ast} = \left\{\left({x}^\ast(t), {x}^\ast(t - 1)\right) : t\in \mathbb{R}\right\} \subset \mathbb{R}^2
    \end{align}
    is a Jordan curve. Moreover, since ${x}^\ast(0)$ is a maximum, the point $({x}^\ast(1), {x}^\ast(0))$ is the maximum in the vertical abscissa of $P\gamma_{\ast}$. Since ${x}^\ast(t)$ is a simple oscillation, the point $({x}^\ast(1), {x}^\ast(0))$ lies above the nullcline $f^{-1}(0) \subset \mathbb{R}^2$. Since the sign of $\partial_2 f$ is fixed, this implies
    \begin{align}
        f({x}^\ast(1), {x}^\ast(0)) \partial_2 f{x}^\ast(0), {x}^\ast(-1) > 0.
    \end{align}
    In particular, since $\ddot x(0) < 0$, we obtain
    \begin{align}
        \begin{split}
            \mathrm{sign}(\dot {x}^\ast(1)) &= -\mathrm{sign}(\dot {x}^\ast(1) \ddot x(0)) \\
            &= -\mathrm{sign}(\dot {x}^\ast(1) r_\ast\partial_2 f({x}^\ast(0), {x}^\ast(-1)) \dot {x}^\ast(-1)) \\
            &= -\mathrm{sign}\left(r_\ast^2 f({x}^\ast(1), {x}^\ast(0)) \partial_2 f({x}^\ast(0), {x}^\ast(-1))\dot {x}^\ast(-1)\right) \\
            &= -\mathrm{sign}(\dot {x}^\ast(-1)),
        \end{split}
    \end{align}
    and \eqref{product-property} follows.
\end{proof}

\begin{remark}
    The intervals ${J}_n$ in \eqref{In-definition} are indexed so that $n$ coincides with the so-called zero number of the respective periodic solution, as defined in \cite{MPSe961}. In particular, the parity of $n$ determines the sign of the delay value $r_\ast$ at which $x^\ast(t)$ solves the monotone delayed feedback DDE \eqref{reference-dde}. If \eqref{reference-dde} has a periodic solution with minimal period ${p_\ast}\in {J}_n$ for $n$ odd, then $r_\ast \partial_2 f$ is negative. Analogously, if $n$ is even, then $r_\ast \partial_2 f$ is positive. 
\end{remark}

By Floquet theory \cite{HaLu93}, the spectrum of the \emph{monodromy operator} 
\begin{align}
    L := \partial_2 S({p_\ast}, {x}^\ast_0; r_\ast),
\end{align}
determines the local stability of $\gamma_{\ast}$. Furthermore, $L$ is the time-${p_\ast}$ solution operator of the DDE initial value problem
\begin{align}
    \begin{split}\label{linearized-equation-periodic-orbit}
        \dot y(t) &= A(t)y(t) + B(t)y(t-1),\\
                y_0 &= \psi,
    \end{split}
\end{align}
with the coefficients
\begin{align}\label{linear-coefficients}
    A(t):=r_\ast\partial_1 f({x}^\ast(t),{x}^\ast(t-1))\quad \text{and}\quad B(t) := r_\ast \partial_2 f({x}^\ast(t),{x}^\ast(t-1)).
\end{align}
The spectrum of $L$ consists of countably many eigenvalues that accumulate to $0$. If $L$ has an eigenvalue $1$ with simple algebraic multiplicity, then the periodic orbit $\gamma_\ast$ is called \emph{hyperbolic}. Otherwise, $L$ has an algebraically double eigenvalue $1$ and we say that $\gamma_\ast$ is \emph{nonhyperbolic}. Proposition \ref{proposition-spectrum} discusses both situations in detail.

\begin{proposition}\label{proposition-spectrum}
    Let ${x}^\ast(t)$ be a normalized periodic solution of the DDE \eqref{reference-dde} at delay $r_\ast \neq 0$ with minimal period $p_\ast$. Then, the monodromy operator $L$ possesses an eigenvalue $\mu_\mathrm{c} > 0$ such that any eigenfunction $\Psi_0 \in \ker(\mu_\mathrm{c} \mathrm{Id} - L)^2$ possesses two zeros on the interval $[0, p_\ast)$. Furthermore, the real generalized eigenspace 
    \begin{align}
        E^\mathrm{c} := \operatorname{Re}\left(\bigoplus_{j \geq 0}\ker\left(\mathrm{Id} - L\right)^j\right) +  \operatorname{Re}\left(\bigoplus_{j \geq 0}\ker\left(\mu_\mathrm{c}\mathrm{Id} - L\right)^j\right)
    \end{align}
    is two-dimensional and $C$ admits an $L$-invariant splitting \begin{align}\label{splitting}
        C = E^\mathrm{c} \oplus R^\mathrm{c},
    \end{align} 
    such that the spectrum of the restriction $L:R^\mathrm{c} \to R^\mathrm{c}$ is disjoint from the annulus
    \begin{align}\label{annulus-equation}
        \left\{z \in\mathbb{C} : \min\{1, \mu_\mathrm{c}\}<|z|<\max\{1, \mu_\mathrm{c}\}\right\} \subset \mathbb{C}.
    \end{align}
    If $\mu_\mathrm{c} = 1$, then the orbit $\gamma_{\ast}$ of ${x}^\ast(t)$ is nonhyperbolic and
    \begin{align}
        E^\mathrm{c} = \operatorname{ker}(\operatorname{Id}-L)^2.
    \end{align}
    Otherwise, $\mu_\mathrm{c} \neq 1$, $\gamma_{\ast}$ is hyperbolic, and 
    \begin{align}
        \dim \operatorname{ker}(\operatorname{Id}-L) = \dim \operatorname{ker}(\mu_\mathrm{c}\operatorname{Id} - L) = 1.
    \end{align}
\end{proposition}
\begin{proof}
    The eigenvalue $\mu_\mathrm{c}$ exists by \cite[Theorem 5.1]{MPNu13}. The subspace $E^\mathrm{c}$ is spanned by the real eigenfunctions associated to the eigenvalue pair $\{1, \mu_\mathrm{c}\}$. By standard theory \cite{HaLu93}, the $L$-invariant subspace $R^\mathrm{c}$ is the range $\operatorname{ran} (\mathrm{Id} - L)^2$ if $\mu_\mathrm{c} = 1$ and $\operatorname{ran}(\mathrm{Id} - L)\cap \operatorname{ran} (\mu_\mathrm{c}\mathrm{Id} - L)$ otherwise.
\end{proof}

Since the eigenvalue $\mu_\mathrm{c}$ in Proposition \ref{proposition-spectrum} characterizes the spectral gap \eqref{annulus-equation} of $L$, we say that it is the \emph{critical eigenvalue} of $\gamma_\ast$. In particular, the critical eigenvalue determines how we can perform a local continuation of $\gamma_\ast$ according to the following lemmata.

\begin{lemma}[Hyperbolic continuation]\label{lemma-beta-hyperbolic-continuation}
    Let ${x}^\ast(t)$ be a normalized hyperbolic periodic solution of the DDE \eqref{reference-dde} at the delay $r_\ast \neq 0$. Then there exist $\varepsilon > 0$, a $C^k$-map $\tilde{p}: (-\varepsilon, \varepsilon) \to (0, \infty)$ and a $C^k$-family of functions $\tilde{x}^{b} \in C^k(\mathbb{R}, \mathbb{R})$ such that the following hold:
    \begin{enumerate}
        \item At ${b} = 0$, we have that
        \begin{align}
            \tilde{p}(0) = p_\ast \quad \text{and} \quad \tilde{x}^0(t) = {x}^\ast(t), \quad \text{for all } t \in \mathbb{R}.
        \end{align}
        \item The functions $\tilde{x}^{b}(t)$ have minimal period $\tilde{p}({b})$, satisfy $\max_t \tilde{x}^{b}(t) = \tilde{x}^{b}(0)$, and solve
        \begin{align}
        \begin{split}
            \dot{\tilde{x}}^{b}(t) &= (r_\ast + {b}) f(\tilde{x}^{b}(t), \tilde{x}^{b}(t - 1)),\quad \text{for all } (t, {b}) \in \mathbb{R} \times (-\varepsilon, \varepsilon).
        \end{split}
        \end{align}
    \end{enumerate}
\end{lemma}

\begin{lemma}[Nonhyperbolic continuation]\label{lemma-beta-r-continuation}
    Let ${x}^\ast(t)$ be a normalized nonhyperbolic periodic solution of the DDE \eqref{reference-dde} at the delay $r_\ast \neq 0$. Then there exist $\varepsilon > 0$, $C^k$-maps $\tilde{p}: (-\varepsilon, \varepsilon) \to (0, \infty)$, $\tilde{r}:(-\varepsilon, \varepsilon) \to \mathbb{R}$, and a $C^k$-family of functions $\tilde{x}^{b} \in C^k(\mathbb{R}, \mathbb{R})$ such that the following hold:
    \begin{enumerate}
        \item At ${b} = 0$, we have that
        \begin{align}
            \tilde{p}(0) = p_\ast,\quad \tilde{r}(0) = r_\ast, \quad \tilde{r}'(0) = 0, \quad \partial_{b} \tilde{x}^0(0) \neq 0, \quad\text{and} \quad \tilde{x}^0(t) = {x}^\ast(t),
        \end{align}
        for all $t \in \mathbb{R}$.
        \item The functions $\tilde{x}^{b}(t)$ have minimal period $\tilde{p}({b})$, satisfy $\max_t \tilde{x}^{b}(t) = \tilde{x}^{b}(0)$, and solve
        \begin{align}
        \begin{split}
            \dot{\tilde{x}}^{b}(t) &= \tilde{r}({b}) f(\tilde{x}^{b}(t), \tilde{x}^{b}(t - 1)),\quad \text{for all } (t, {b}) \in \mathbb{R} \times (-\varepsilon, \varepsilon).
        \end{split}
        \end{align}
    \end{enumerate}
\end{lemma}

\begin{proof}[Proof of Lemma \ref{lemma-beta-hyperbolic-continuation}] 
    Let us consider the extended DDE \eqref{extended-DDE} and notice that the solution semiflow of \eqref{extended-DDE} with initial condition $(\varphi; r) \in C \times \mathbb{R}$ is $(S(t, \varphi; r); r)$. Next, we choose a section transverse to the extended orbit $(\gamma_{\ast}; r_\ast)$ at $({x}^\ast_0; r_\ast)$. Since we assumed that that the solution ${x}^\ast(t)$ is normalized, ${x}^\ast(0)$ is a maximum and we choose $(\Sigma; r_\ast)  \subset C\times \mathbb{R}$, where
    \begin{align} \label{poincare-section}
        \Sigma := \{\varphi \in C : f \circ P(\varphi) = 0\}.
    \end{align}
    By Lemma \ref{lemma-basic-periodic}, we have $\ddot {x}^\ast(0) \neq 0$ and we can define the normalized differential
    \begin{align}
        \begin{split}
            \, \mathrm{d} {\Sigma}(\varphi) :=& \frac{r_\ast}{\ddot {x}^\ast(0)}(\nabla f({x}^\ast(0), {x}^\ast(-1)) \cdot P \varphi)\\
            =& \frac{r_\ast}{\ddot {x}^\ast(0)}\left(\partial_1 f({x}^\ast(0), {x}^\ast(-1))\varphi(0) + \partial_2 f({x}^\ast(0), {x}^\ast(-1))\varphi(-1)\right),
        \end{split}
    \end{align}
    so that the tangent space of $\Sigma$ at ${x}^\ast_0$ satisfies
    \begin{align}
        T_{{x}^\ast_0} \Sigma = \mathrm{ker}\left(\, \mathrm{d} {\Sigma}\right).
    \end{align}
    By construction, we have that $\, \mathrm{d} {\Sigma}(\dot {x}^\ast_0) = 1$ and $\ker (\, \mathrm{d} {\Sigma}) + \operatorname{span}_{\mathbb{R}}\{\dot {x}^\ast_0\} = C$ which ensures that $(\Sigma; r_\ast)$ is transverse to $(\gamma_{\ast}; r_\ast)$ at $({x}^\ast_0; r_\ast)$. Hence, using the implicit function theorem, we define the Poincar\'{e} time, that is, the first return time to $\Sigma$ as the local function $\mathcal{T}: C \times \mathbb{R} \to \mathbb{R}$ that solves 
    \begin{align}\label{poincare-time}
        (f \circ P)\left(S(\mathcal{T}(\varphi; r), \varphi; r)\right) = 0\quad \text{and} \quad \mathcal{T}({x}^\ast_0; r_\ast) = {p_\ast}.
    \end{align}
    The Poincar\'{e} map is $\mathcal{S} : \Sigma \times \mathbb{R} \to \Sigma \times \mathbb{R}$ given by
    \begin{align}\label{poincare-maps}
        \mathcal{S}(\varphi; r) := (S(\mathcal{T}(\varphi; r), \varphi; r); r).
    \end{align}
    The domain of definition of $\mathcal{S}$ in \eqref{poincare-maps} coincides with that of the Poincar\'{e} time \eqref{poincare-time}. If we denote the projection onto $T_{{x}^\ast_0}\Sigma$ along $\dot {x}^\ast_0$ by
    \begin{align}
        \begin{array}{rcl}\label{sigma-projection}
            Q_{\Sigma} : C & \longmapsto & T_{{x}^\ast_0}\Sigma \\
            \varphi & \longmapsto & \varphi - \, \mathrm{d} {\Sigma}(\varphi)\dot {x}^\ast_0,
        \end{array}
    \end{align}
    then the Fr\'{e}chet derivative of the Poincar\'{e} map \eqref{poincare-maps} is
    \begin{align}\label{def-ext-mon}
        \begin{array}{rcl}
            \mathcal{L} := D\mathcal{S}({x}^\ast_0; r_\ast): T_{{x}^\ast_0}\Sigma \times \mathbb{R} & \longrightarrow & T_{{x}^\ast_0}\Sigma \times \mathbb{R} \\
            (\psi; r) & \longmapsto & ((Q_{\Sigma} L) \psi + r Q_{\Sigma} \varrho_{p_\ast}; r).
        \end{array}
    \end{align}
    Here $\varrho(t) := \partial_3 S(p_\ast, x_0^\ast; r_\ast)$ is the solution of the initial value problem
    \begin{align}\label{IVP-rho}
    \begin{split}
        \dot \varrho(t) = A(t)\varrho(t) + B(t)\varrho(t-1) + \frac{1}{r_\ast}\dot x^\ast(t), \quad \varrho_0 = 0,
    \end{split}
    \end{align}
    with coefficients given by \eqref{linear-coefficients}. Since the spectrum of $L$ consists of eigenvalues and $0$, so does the spectrum of $\mathcal{L}$. Moreover, if $\varphi$ is an eigenfunction of $L$, then $(Q_{\Sigma}\varphi; 0) \in T_{{x}^\ast_0}\Sigma \times \mathbb{R}$ is an eigenfunction of $\mathcal{L}$ associated to the same eigenvalue. Notice that the Poincar\'{e} map \eqref{poincare-maps} is invertible in the $\Sigma$-component because $Q_{\Sigma} \dot {x}^\ast_0 = 0$ and we assumed that $\gamma_\ast$ is hyperbolic. Therefore, the implicit function theorem yields a unique solution of 
    \begin{align}
        \mathcal{S}\left(\tilde{x}^{b}_0; r_\ast + {b}\right), \quad \text{and} \quad \tilde{x}^0_0 = {x}^\ast_0,
    \end{align}
    for all ${b} \in (-\varepsilon, \varepsilon)$. Since we are performing the continuation on the section $\Sigma$ given by \eqref{poincare-section}, the identity $\max_t \tilde{x}^{b}(t) = \tilde{x}^{b}(0)$ follows by construction. The minimal period of each solution is given by the Poincar\'{e} time \eqref{poincare-time} via
    \begin{align}
        \tilde{p}({b}) := \mathcal{T}\left(\tilde{x}_0^{b}; r_\ast + {b}\right).
    \end{align}
    Since all the maps we have considered inherit the $C^k$-regularity from the implicit function theorem, the proof is complete.
\end{proof}

Before proving Lemma \ref{lemma-beta-r-continuation}, we need to introduce a projection that plays a crucial role. The proof of the following lemma is given in Section \ref{Sec7} to ease the technical burden.
\begin{lemma}\label{lemma-projections}
    Let $x^\ast(t)$ be a normalized periodic solution of the DDE \eqref{reference-dde} at delay $r_\ast \neq 0$ with monodromy operator $L$. Let $p_\ast$ denote the minimal period, and consider the function $\varrho(t) = \partial_3 S(t, x_0^\ast; r_\ast)$ solving \eqref{IVP-rho}. If $\Psi_0$ is a generalized eigenfunction of $L$ such that $E^\mathrm{c} = \mathrm{span}_{\mathbb{R}}\{\dot x_0^\ast, \Psi_0\}$, then there exists a projection $P_\Psi : C \to C$ with $Q_\Psi := \mathrm{Id} - P_\Psi$ such that 
    \begin{align}
        \operatorname{ran} P_\Psi = \mathrm{span}_{\mathbb{R}}\{\Psi_0\} \quad \text{and} \quad \operatorname{ran} Q_\Psi = \mathrm{span}_{\mathbb{R}}\{\dot x_0^\ast\} \oplus R^{\mathrm{c}}.
    \end{align}
    Moreover, $P_\Psi$ and $Q_\Psi$ satisfy
    \begin{align}
        P_\Psi \dot x_0^\ast = 0, \quad Q_\Psi \Psi_0 = 0, \quad and \quad P_\Psi \varrho_{p_\ast} \neq 0.
    \end{align}
\end{lemma}
With this we can finally prove the nonhyperbolic continuation in Lemma \ref{lemma-beta-r-continuation}.
\begin{proof}[Proof of Lemma \ref{lemma-beta-r-continuation}]
     We discuss the Poincar\'{e} map \eqref{poincare-maps} in the nonhyperbolic situation $\mu_\mathrm{c} = 1$. Notice that $\mathcal{L} = D\mathcal{S}({x}^\ast_0; r_\ast)$ given by \eqref{def-ext-mon} possesses mixed eigenfunctions $(\zeta_0; 1) \in T_{{x}^\ast_0}\Sigma \times \mathbb{R}$ that solve the eigenvalue problem
    \begin{align}\label{mixed-eigenvalue-problem}
        ((Q_{\Sigma} L)\zeta_0 + Q_{\Sigma} \varrho_{p_\ast}; 1) = \mu(\zeta_0; 1), \quad \text{for some }\mu \in \mathbb{C}.
    \end{align}
    Here $\varrho_{p_\ast} = \partial_3 S(p_\ast, x_0^\ast; r_\ast)$ solves \eqref{IVP-rho}. Comparing the second components in \eqref{mixed-eigenvalue-problem}, the only mixed eigenvalue is $\mu = 1$. Moreover, let us choose a nonzero function $\Psi_0 \in Q_{\Sigma}E^\mathrm{c}$ and consider the projections in Lemma \ref{lemma-projections}. Then the restricted map $({Q}_{\Psi}; \mathrm{Id}) \circ \mathcal{L}: ({Q}_{\Psi} T_{{x}^\ast_0} \Sigma) \times \mathbb{R} \to ({Q}_{\Psi} T_{{x}^\ast_0} \Sigma) \times \mathbb{R}$ is invertible. Hence, there exists a unique solution $(\zeta_0; 1)$ of \eqref{mixed-eigenvalue-problem} with $\zeta_0 \in \ker P_\Psi$ that satisfies
    \begin{align}
        \mathcal{L} (\zeta_0; 1) = (\zeta_0; 1) + P_\Psi ((Q_\Sigma L) \zeta_0 + Q_\Sigma \varrho_{p_\ast}; 1).
    \end{align}
    By Lemma \ref{lemma-projections} and the definition \eqref{sigma-projection}, we obtain
    \begin{align}\label{eq-nonzero-proj}
        \begin{split}
            P_\Psi ((Q_\Sigma L) \zeta_0 + Q_\Sigma \varrho_{p_\ast}) &= P_\Psi(L\zeta_0 + \varrho_{p_\ast})\\
            &= P_\Psi \varrho_{p_\ast}\\
            &\neq 0,
        \end{split}
    \end{align}
    which yields $(\zeta_0; 1)$ in $\ker(\mathcal{L}^2 - \mathrm{Id}) \setminus \ker(\mathcal{L} - \mathrm{Id})$.
    We prove that this ensures a continuation of periodic solutions with respect to the coordinate in the direction $(\Psi_0; 0)$.
    Indeed, we have shown that $\mathcal{L}$ has the same spectrum as $L$ and the eigenvalues have the same algebraic multiplicity. In particular, if $\mu_\mathrm{c} = 1$, then
    \begin{align}\label{def-gap}
        \ker (\mathcal{L} - \mathrm{Id})^2 = \mathrm{span}_{\mathbb{R}} \{(\Psi_0; 0), (\zeta_0; 1)\},
    \end{align}
    and the remainder of the spectrum is disjoint from the complex unit circle.

    Following \cite{HiPuSh77}, by the spectral gap \eqref{def-gap}, all fixed points of
    \begin{align}\label{fixed-point}
        \mathcal{S}(\varphi; r) = (\varphi; r),
    \end{align}
    sufficiently close to $(x_0^\ast; r_\ast)$ belong to a local center manifold $W^\mathrm{c}_{\mathrm{loc}}$. Moreover, the center manifold is represented locally by a $C^k$-function $h^\mathrm{c} : \mathbb{R}^2 \to {Q}_{\Psi} (T_{{x}^\ast_0} \Sigma)$ such that $h^\mathrm{c}(0; r_\ast) = {x}^\ast_0$, $Dh(0; r_\ast) = 0$. Hence, sufficiently close to $({x}^\ast_0; r_\ast)$, all fixed points of \eqref{fixed-point} are of the form 
    \begin{align}\label{expansion-nonhyp-solution}
        (\varphi({b}; r); r) := ({x}^\ast_0 + {b} \Psi_0 + (r - r_\ast) \zeta_0 + h^\mathrm{c}({b}; r); r) \in C \times \mathbb{R}.
    \end{align}
    
    Recall from our construction that $\zeta_0, h^\mathrm{c} \in {Q}_{\Psi} (T_{{x}^\ast_0} \Sigma)$. Hence, if we denote the first component of the Poincar\'{e} map $\mathcal{S}$ by $\mathcal{S}^1$, then the problem \eqref{fixed-point} is equivalent to solving
    \begin{align}
        P_\Psi \left(\mathcal{S}^1(\varphi({b};r);r) - \varphi({b};r)\right) = 0, \quad \varphi(0;r_\ast) = {x}^\ast_0.
    \end{align}
    To apply the implicit function theorem, by \eqref{eq-nonzero-proj}, notice that
    \begin{align}
        \begin{split}
            \partial_r {P}_\Psi \left(
            \mathcal{S}^1(\varphi({b}; r)) - \varphi({b}; r)\right)&= {P}_\Psi \left( {Q}_\Sigma L\zeta_0 + {Q}_\Sigma \varrho_{p_\ast} - \zeta_0 \right)\\
            &= {P}_\Psi\varrho_{p_\ast}\\
            &\neq 0.
        \end{split}
    \end{align}
    Hence, there exists a unique local $C^k$-continuation such that
    \begin{align}\label{fixed-point-Psi}
        {P}_\Psi \left(\mathcal{S}^1(\varphi({b};\tilde{r}({b}));r({b})) - \varphi({b};\tilde{r}({b}))\right) = 0, \quad (\varphi(0;\tilde{r}(0)); \tilde{r}(0)) = ({x}^\ast_0; r_\ast),
    \end{align}
    and we use the Poincar\'{e} time \eqref{poincare-time} to define
    \begin{align}
        \tilde{x}_0^{b} := \varphi({b}; \tilde{r}({b})) \quad \text{and} \quad \tilde{p}({b}) := \mathcal{T}(\tilde{x}_0^{b}; \tilde{r}({b})).
    \end{align}\\
    Finally, we show that $\tilde{r}'(0) = 0$ and $\partial_{b} \tilde{x}^0(0) \neq 0$. Indeed, differentiating the continuation equation \eqref{fixed-point-Psi} in ${b}$ yields
    \begin{align}
        \begin{split}
            0 &= \tilde{r}'(0) {P}_\Psi  \varrho_{p_\ast},
        \end{split}
    \end{align}
    which ensures $\tilde{r}'(0) = 0$, as claimed. Next, suppose that $\partial_{b}\tilde{x}^0(0) = 0$, by construction, we have that
    \begin{align}
        \partial_{b} \tilde{x}^0_0 = \Psi_0 + r'(0) \zeta_0 = \Psi_0.
    \end{align}
    Hence, $\Psi_0(0) = 0$ and we can find a $\kappa \in \mathbb{R}$ such that $\dot x_0^\ast + \kappa \Psi_0$ is an eigenfunction of $L$ with a double zero at $0$. This is a contradiction to \cite[Theorem 5.1]{MPNu13}, since the eigenfunctions of $L$ possess only simple zeros.
\end{proof}

\begin{remark}
    A minor inconvenience of using the phase space $C$ is that it does not admit the smooth cutoff functions required for constructing $W^\mathrm{c}_{\mathrm{loc}}$ in \cite{HiPuSh77}. In practice, this is not an issue because the solutions of the fixed point problem \eqref{fixed-point} belong to the Sobolev space of functions with a square-integrable weak derivative. Such Sobolev space allows smooth cutoff functions, so that we may construct the center manifold in Sobolev space by replicating the Poincar\'{e} map construction in the proof of Lemma \ref{lemma-beta-hyperbolic-continuation}. Further technical details on DDE semiflows defined in Sobolev spaces can be found in \cite{Ni19, Lo23}.
\end{remark}

\section{Proof of theorems \ref{thm1}, \ref{thm2}, and \ref{thm3}} \label{Sec4}
First, we show that, locally, we can replace the parameter $b$ in Lemma \ref{lemma-beta-hyperbolic-continuation} and Lemma \ref{lemma-beta-r-continuation} by the amplitude of the periodic solution with initial condition $(\tilde{x}^b_0; \tilde{r}(b))$. Achieving this is crucial because the amplitude is a global parameter that allows us to compare local charts $\beta_{\mathrm{loc}}$ of the branch $\mathcal{B}$, in contrast to the local parameter $b$.
\begin{theorem}\label{thm-annulus-embedding}
    Let ${x}^\ast(t)$ be a periodic solution of the DDE \eqref{reference-dde} at delay $r_{\ast}$ with minimal period $p_\ast$. Furthermore, consider $\tilde{x}^{b}_t$, $\tilde{r}(b)$, $\tilde{p}({b})$, the ${b}$-continuation of periodic orbits obtained in Lemma \ref{lemma-beta-hyperbolic-continuation} if the orbit $\gamma_{\ast}$ of ${x}^\ast(t)$ is hyperbolic and consider the continuation in Lemma \ref{lemma-beta-r-continuation} otherwise. Then, there exists an $\varepsilon>0$ such that the map $\tilde{G}: \mathbb{R}^2 \to \mathbb{R}^2$ given by
    \begin{align}\label{definition-annulus-embedding}
    \tilde{G}(t, {b}) := \left(\tilde{x}^{b}(t), \tilde{x}^{b}(t - 1)\right)
    \end{align}
    is a $C^k$-embedding of the annulus 
    \begin{align}\label{domain-equation-annulus}
        \tilde{\mathbb{A}} := \{(t; {b}): t\in \mathbb{R} / \tilde{p}({b})\mathbb{Z},\; {b} \in (-\varepsilon, \varepsilon)\}.
    \end{align}
\end{theorem}

Before introducing the lemmata required to prove Theorem \ref{thm-annulus-embedding}, we show a corollary. More precisely, we can  locally replace the parameter $b$ in Lemma \ref{lemma-beta-hyperbolic-continuation} and Lemma \ref{lemma-beta-r-continuation} by the amplitude $a$ of the periodic solutions.

\begin{corollary}[Local amplitude continuation]\label{corollary-local-amplitude-continuation}
    Let ${x}^\ast(t)$ be a normalized periodic solution of the DDE \eqref{reference-dde} at the delay $r_\ast \neq 0$. Then there exist an open interval $\mathcal{J}\subset \mathbb{R}$ containing ${x}^\ast(0)$, $C^k$-maps $p: \mathcal{J} \to (0, \infty)$, $r:\mathcal{J} \to \mathbb{R}$, and a $C^k$-family of functions $x^a \in C^k(\mathbb{R}, \mathbb{R})$ such that the following hold:
    \begin{enumerate}
        \item At $a_\ast := {x}^\ast(0)$, we have that
        \begin{align}
            p(a_\ast) = p_\ast,\quad r(a_\ast) = r_\ast, \quad \text{and} \quad x^{a_\ast}(t) = {x}^\ast(t), \quad \text{for all } t \in \mathbb{R}.
        \end{align}
        \item The functions $x^a(t)$ have minimal period $p(a)$, satisfy $\max_t x^a(t) = a$, and solve
        \begin{align}
        \begin{split}
            \dot x^a(t) &= r(a) f(x^a(t), x^a(t - 1)),\quad \text{for all } (t, a) \in \mathbb{R} \times \mathcal{J}.
        \end{split}
        \end{align}
    \end{enumerate}
    Moreover, let $\mathbb{A}_{\mathrm{loc}} := \{(t, a) : t \in \mathbb{R} / p(a)\mathbb{Z} ,\, a \in \mathcal{J}\}$, then the transformations
    \begin{align}
        \begin{array}{rcl}
            \beta_{\mathrm{loc}}: \mathbb{A}_{\mathrm{loc}} & \longrightarrow & C \times \mathbb{R}  \\
            (t, a) & \longmapsto & (x_t^a; r(a)),
        \end{array}
    \end{align}
    and
    \begin{align}
        \begin{array}{rcl}
            G_{\mathrm{loc}}: \mathbb{A}_{\mathrm{loc}} & \longrightarrow & C \times \mathbb{R}  \\
            (t, a) & \longmapsto & (x^a(t), x^a(t - 1)),
        \end{array}
    \end{align}
    are $C^k$-embeddings.
\end{corollary}
\begin{proof}
    As a consequence of Theorem \ref{thm-annulus-embedding}, we have that the amplitude $a({b}) := \tilde{x}^{b}(0)$ of the ${b}$-continuations in Lemma \ref{lemma-beta-hyperbolic-continuation} and Lemma \ref{lemma-beta-r-continuation} satisfies
    \begin{align}
        \partial_{b} \tilde{x}^{b}(0) \neq 0.
    \end{align}
    Hence, locally, we can write ${b}$ as a $C^k$-function of the amplitude, independently of whether the orbit $\gamma_{\ast}$ of ${x}^\ast(t)$ is hyperbolic or not. To see that $\beta_{\mathrm{loc}}$ is an embedding, notice that $G_{\mathrm{loc}}(t, a) = \tilde{G}(t, b(a))$ is an embedding. Hence, by the commutative diagram
    \begin{equation}\label{local-cd}
    \begin{tikzcd}
      \mathbb{A}_{\mathrm{loc}} \arrow[r, "\beta_{\mathrm{loc}}"] \arrow[dr, "G_{\mathrm{loc}}"'] 
        & \beta_{\mathrm{loc}}(\mathbb{A}_{\mathrm{loc}}) \arrow[d, "\bar{P}"] \\
        & \bar{P}\beta_{\mathrm{loc}}(\mathbb{A}_{\mathrm{loc}}),
    \end{tikzcd}
    \end{equation}
    we have in \eqref{local-cd} that $\bar{P}$ and $\beta_{\mathrm{loc}}$ are $C^k$-diffeomorphisms with inverses
    \begin{align}
        \bar{P}^{-1} = \beta_{\mathrm{loc}} \circ G_{\mathrm{loc}}^{-1} \quad \text{and} \quad \beta_{\mathrm{loc}}^{-1} = G_{\mathrm{loc}}^{-1} \circ \bar{P}.
    \end{align}
\end{proof}

\begin{lemma}\label{lemma-annulus-homeomorphism}
    Under the assumptions of Theorem \ref{thm-annulus-embedding}, the map $\tilde{G}$ given by \eqref{definition-annulus-embedding} is a $C^0$-embedding of the annulus.
\end{lemma}
Proving Lemma \ref{lemma-annulus-homeomorphism} requires a discussion on intersections of continuous families of Jordan curves. For this reason, we have delayed the proof to Section \ref{Sec5}.

\begin{lemma}\label{lemma-not-identically-zero}
    Under the assumptions of Theorem \ref{thm-annulus-embedding}, there exists a $t_0 \in \mathbb{R}$ such that
    \begin{align}
        \det (D\tilde{G}(t_0, 0)) \neq 0.
    \end{align}
\end{lemma}
\begin{proof}
    We consider the nonhyperbolic and hyperbolic situations separately.
    First, if the orbit of $x^\ast(t)$ is nonhyperbolic, then, by Lemma \ref{lemma-beta-r-continuation}, we have that $\partial_{b} \tilde{x}^0(0) \neq 0$. Since
    \begin{align}
        \det (D\tilde{G}(0, 0)) = \det\left(\begin{array}{cc}
            0 & \dot{\tilde{x}}^0(-1) \\
            \partial_{b} \tilde{x}^0(0) & \partial_{b} \tilde{x}^0(-1)
        \end{array} \right),
    \end{align}
    and $\dot{\tilde{x}}^0(-1) \neq 0$ by Lemma \ref{lemma-basic-periodic}, the claims hold with $t_0 = 0$.
    
    Second, we assume that the orbit of ${x}^\ast(t)$ is hyperbolic. Notice that $\dot{\tilde{x}}^0(t)$ and $\partial_{b} \tilde{x}^0(t)$ solve the DDEs
    \begin{align}\label{equation-t-beta-derivative}
        \begin{split}
            \ddot{\tilde{x}}^0(t) &= A(t) \dot{\tilde{x}}^0(t) + B(t)  \dot {\tilde{x}}^0(t - 1),\; \text{and}\\
            \partial_{b} \dot{\tilde{x}}^0(t) &= A(t) \partial_{b} \tilde{x}^0(t) + B(t) \partial_{b} \tilde{x}^0(t - 1) + \frac{1}{r_\ast} \dot{\tilde{x}}^0(t),
        \end{split}
    \end{align}
    with coefficients \eqref{linear-coefficients}. We proceed by contradiction and suppose that $\det (D\tilde{G}(t, 0)) \equiv 0$. Thus, there exists a continuous function $\lambda(t)$ such that
    \begin{align}
        \lambda(t) = \lambda(t - 1)\quad \text{and} \quad \partial_{b} \tilde{x}^0(t) = \lambda(t)\dot{\tilde{x}}^0(t).
    \end{align}
    On the one hand, if the minimal period $p_\ast$ is irrational, then $\lambda(t)$ is constant on a dense subset of $\mathbb{R}/p_\ast\mathbb{Z}$ and we obtain that $\lambda(t) \equiv \lambda$ for some $\lambda\in \mathbb{R}$. Hence $\partial_{b} \tilde{x}^0(t) = \lambda \dot{\tilde{x}}^0(t)$, in contradiction to the identities \eqref{equation-t-beta-derivative}.
    On the other hand, if $p_\ast$ is rational, then, by \eqref{equation-t-beta-derivative}, we obtain
    \begin{align}
        \begin{split}
            \partial_{b} \dot{\tilde{x}}^0(t) &= \lambda(t) \ddot{\tilde{x}}^0(t) + \dot\lambda(t) \dot{\tilde{x}}^0(t)\\
            &= \lambda(t) \ddot{\tilde{x}}^0(t) + \frac{1}{r_\ast} \dot{\tilde{x}}^0(t).
        \end{split}
    \end{align}
    In particular, we can integrate $\dot\lambda(t)$ to derive
    \begin{align}
        \lambda(t) = \left(\frac{1}{r_\ast}t + \lambda(0)\right), 
    \end{align}
    which implies that $\lambda(t)$ is constant. Hence, the equations \eqref{equation-t-beta-derivative} imply $\dot{\tilde{x}}^0(t) \equiv 0$ and we have reached a contradiction.
\end{proof}

\begin{lemma}[Propagation of singularities]\label{lemma-propagate-zeros}
    Under the assumptions of Theorem \ref{thm-annulus-embedding}, if $\det (D\tilde{G}(t, 0)) = 0$ for some $t\in \mathbb{R}$, then $\det(D\tilde{G}(t - m, 0)) = 0$ for all $m \in \mathbb{N}$.
\end{lemma}
\begin{proof}
    \textit{Step 1:}
    If $\det (D\tilde{G}(t, 0)) = 0$, then there exists $\lambda \in \mathbb{R}$ such that
    \begin{align}\label{eq-parallel-derivatives}
        (\partial_{b} \tilde{x}^0(t), \partial_{b} \tilde{x}^0(t - 1)) = \lambda (\dot{\tilde{x}}^0(t), \dot{\tilde{x}}^0(t - 1)),
    \end{align}
    and $\dot{\tilde{x}}^0(t)(\partial_{b} \tilde{x}^0(t - 2) - \lambda \dot{\tilde{x}}^0(t - 2)) = 0$.
    
    Indeed, notice that
    \begin{align}\label{definition-determinant-DG}
        D\tilde{G}(t, 0) = \left(
            \begin{array}{cc}
               \dot{\tilde{x}}^0(t)  & \dot{\tilde{x}}^0(t - 1) \\
               \partial_{b} \tilde{x}^0(t) & \partial_{b} \tilde{x}^0(t - 1)
            \end{array}\right),
    \end{align}
    therefore, $\det (D\tilde{G}(t, 0)) = 0$ implies that the $t$- and ${b}$-derivatives of $(\tilde{x}^{b}(t), \tilde{x}^{b}(t - 1))$ are parallel as in \eqref{eq-parallel-derivatives}. Moreover, by Lemma \ref{lemma-annulus-homeomorphism}, $\tilde{G}$ is a homeomorphism and any singularity of $\det(D\tilde{G}(t;0))$ must be a tangency, i.e., $\nabla \det (D\tilde{G}(t, 0)) = 0$. In particular, using \eqref{equation-t-beta-derivative}, the condition $\partial_t \det \tilde{G}(t_0, 0) = 0$ translates into
    \begin{align}
        \begin{split}
            \partial_t\det (D\tilde{G}(t; 0)) &= B(t - 1) \dot{\tilde{x}}^0(t) (\partial_{b} \tilde{x}^0(t - 2) - \lambda\dot{\tilde{x}}^0(t - 2)) = 0,
        \end{split}
    \end{align}
    with $B(t) \neq 0$ given by \eqref{linear-coefficients}, which proves \textit{Step 1}. 
    
    Notice that, by \textit{Step 1}, any singularity propagates backwards so that $\det(D\tilde{G}(t - m, 0)) = 0$ as long as $\dot{\tilde{x}}^0(t - n) \neq 0$ for all $n \leq m$. In \textit{Step 2}, we show that the sequence of singularities extends past time-extrema of the solution.
    
    \textit{Step 2:} If $\dot{\tilde{x}}^0(t) = 0$ and $\det (D\tilde{G}(t, 0)) = 0$, then $\det (D\tilde{G}(t- 1, 0)) = 0$.
    
    Let us assume that $\det (D\tilde{G}(0, 0)) = 0$, then we show that $\det (D\tilde{G}(-1, 0)) = 0$. The situation at the time of depth $\det (D\tilde{G}(q, 0)) = 0$ is analogous. Since $\dot{\tilde{x}}^0(0) = 0$, it follows from $\det(D\tilde{G}(0, 0)) = 0$ that $\partial_{b} \tilde{x}^0(0) = 0$. Moreover, since
    \begin{align}
        \begin{split}
            \det (D\tilde{G}(0, 0)) &= \det \left(
        \begin{array}{cc}
            0 & \dot{\tilde{x}}^0(-1) \\
            0 & \partial_{b} \tilde{x}^0(-1)
        \end{array}
        \right) \\
        &= \det \left(
        \begin{array}{cc}
            \dot{\tilde{x}}^0(1) & 0 \\
            \partial_{b} \tilde{x}^0(1) & 0
        \end{array}
        \right)\\
        &= \det(D\tilde{G}(1, 0))
        \end{split}
    \end{align}
    we immediately obtain that $\det (D\tilde{G}(1, 0)) = 0$. Since $\tilde{x}^0(0)$ is a maximum, we have the expansion
    \begin{align}
        a({b}) = \partial_{b}^2 \tilde{x}^0(0) \frac{{b}^2}{2} + O({b}^3), \quad \text{as } b \to 0.
    \end{align}
    It follows from Lemma \ref{lemma-annulus-homeomorphism} that the amplitude $a(b)$ is locally $C^0$-invertible, which requires $\partial_{b}^2 \tilde{x}^0(0) = 0$. We recall from \textit{Step 1} that $\nabla\det(D\tilde{G}(0, 0)) = 0$, therefore,
    \begin{align}
        \begin{split}
            \partial_{b} \det (D\tilde{G}(0, 0)) &= -\dot{\tilde{x}}^0(1) (\lambda^2 B(0)\dot{\tilde{x}}^0(- 1) - \partial_{b}^2 \tilde{x}^0(0)) \\
            &= -\dot{\tilde{x}}^0(1) \lambda^2 B(0) \dot{\tilde{x}}^0(-1)\\
            &= 0.
        \end{split}
    \end{align}
    By the inequality \eqref{product-property}, we have that $\dot{\tilde{x}}^0(1) \dot{\tilde{x}}^0(-1) < 0$, hence, we conclude that $\lambda = 0$ in \textit{Step 1}, which implies 
    \begin{align}
        \partial_{b} \tilde{x}^0(1) = \partial_{b} \tilde{x}^0(0) = \partial_{b} \tilde{x}^0(-1) = 0.
    \end{align}
    Finally, we use \eqref{equation-t-beta-derivative} to expand $\det(D\tilde{G}(1 + t, 0))$ and $\det(D\tilde{G}(t, 0))$ in $t$, so that
    \begin{align}
        \begin{split}
            \det(D\tilde{G}(1 + t, 0)) &= B(0) \left(\partial_{b} \dot{\tilde{x}}^0(-1) - \frac{\tilde{r}'(0)}{\tilde{r}_\ast} \dot{\tilde{x}}^0(-1)\right)\dot{\tilde{x}}^0(1)\frac{t^2}{2} + O(t^3),
        \end{split}\\
        \begin{split}
            \det(D\tilde{G}(t, 0)) &= B(0) \left(\partial_{b} \dot{\tilde{x}}^0(-1) - \frac{\tilde{r}'(0)}{\tilde{r}_\ast} \dot{\tilde{x}}^0(-1)\right)\dot{\tilde{x}}^0(-1)\frac{t^2}{2} + O(t^3),
        \end{split}
    \end{align}
    as $t \to 0$. By the inequality \eqref{product-property}, for small $t$, we conclude that $\det (D\tilde{G}(1 + t, 0))$ and $\det (D\tilde{G}(t, 0))$ have opposite signs unless
    \begin{align}
        \partial_{b} \dot{\tilde{x}}^0(-1) - \frac{\tilde{r}'(0)}{\tilde{r}_\ast} \dot{\tilde{x}}(-1) = B(-1) \partial_{b} \tilde{x}^0(-2) = 0.
    \end{align}
    Thus, we obtain $\partial_{b} \tilde{x}^0(-2) = \partial_{b} \tilde{x}^0(-1) = 0$ and $\det D\tilde{G}(-1, 0) = 0$, as claimed.
    
\end{proof}

\begin{proof}[Proof of Theorem \ref{thm-annulus-embedding}]
    We proceed by contradiction. Suppose that $\det (D\tilde{G}(t_0,0)) = 0$ for some $t_0\in \mathbb{R}$. By Lemma \ref{lemma-propagate-zeros}, we have that $\det (D\tilde{G}(t_0 - n, 0)) = 0$ for all $n \in \mathbb{N}_0$. We consider two scenarios. First, if the minimal period $p_\ast$ of $\tilde{x}^0(t)$ is irrational, then the points $\{(\tilde{x}^0(t_0 - n), \tilde{x}^0(t_0 - n - 1))\}_{n\in \mathbb{N}_0}$ are dense on $P\tilde{\gamma}_0$ and, by continuity, we have that $\det(D\tilde{G}(t, 0)) \equiv 0$, in contradiction to Lemma \ref{lemma-not-identically-zero}. Second, if $p_\ast$ is rational, then the set $\{(\tilde{x}^0(t_0 - n), \tilde{x}^0(t_0 - n - 1))\}_{n\in \mathbb{N}_0}$ is finite. Hence, there exist $m, M\in\mathbb{N}$ such that $\tilde{x}^0(t - M) = \tilde{x}^0(t - mp) = \tilde{x}^0(t)$. In particular, the $M$-dimensional vector
    \begin{align}
        u^j(t) := \partial_{b} \tilde{x}^0(t - j), \quad j = 0, \dots, M - 1,
    \end{align}
    solves the initial value problem 
    \begin{align}\label{ODE-IVP}
        \begin{split}
            \dot u^j(t) &= A(t - j) u^j(t) + B(t - j)u^{j + 1}(t) + \frac{\tilde{r}'(0)}{r_{\ast}}\dot{\tilde{x}}^0(t - j),\quad j \mod M,\\
            u^j(t_0) &= \lambda \partial_{b} \tilde{x}^0(t_0 - j),
        \end{split}
    \end{align}
    with coefficients \eqref{linear-coefficients}. However, direct substitution shows that \eqref{ODE-IVP} is solved uniquely, by
    \begin{align}
        u^j(t) = \left(\frac{\tilde{r}'(0)}{r_{\ast}}(t - t_0) + \lambda\right)\dot{\tilde{x}}^0(t - j),\quad j \mod M,
    \end{align}
    following the argument above, we conclude that
    \begin{align}
        \left(\partial_{b} \tilde{x}^0(t), \partial_{b} \tilde{x}^0(t - 1)\right) = \left(\frac{\tilde{r}'(0)}{r_{\ast}}(t - t_0) + \lambda\right)\left(\dot{\tilde{x}}^0(t), \dot{\tilde{x}}^0(t - 1)\right).
    \end{align}
    Therefore, $\det \tilde{G}(t, 0)\equiv 0$, in contradiction to Lemma \ref{lemma-not-identically-zero}.
\end{proof}

\begin{proof}[Proof of Theorem \ref{thm1} and Theorem \ref{thm2}]
    Consider a point $({x}^\ast_0; r_\ast) \in \mathcal{B}$ such that $a_\ast := {x}^\ast(0)$ is the amplitude of the periodic solution of \eqref{reference-dde} with initial condition ${x}^\ast_0$ at delay $r_\ast$. Corollary \ref{corollary-local-amplitude-continuation} shows that there exists a local time-amplitude chart $\beta_{\mathrm{loc}}$ of $\mathcal{B}$. Moreover, the domain of the map $\beta_{\mathrm{loc}}$ can always be enlarged if any of the boundaries belongs to $\mathcal{B}$. We define the map $\beta$ in Theorem \ref{thm2} to be the maximal extension of $\beta_{\mathrm{loc}}$. By construction, $\beta$ is $C^k$-differentiable, and $\operatorname{ran} \beta$ is both open and closed in $\mathcal{B}$. Since $\mathcal{B}$ is connected, we conclude that $\beta$ is surjective.

    Finally, define $G(t, a) := \bar{P}\beta(t; a) = Px_t^a$ and consider the commutative diagram
    \begin{equation}\label{cd-thm1}
        \begin{tikzcd}
        \mathbb{A} \arrow[r, "\beta"] \arrow[dr, "G"'] & \mathcal{B} \arrow[d, "\bar{P}"] \\
         & \mathcal{O}.
        \end{tikzcd}
    \end{equation}
    We highlight that the proof of Theorem \ref{thm-annulus-embedding} is independent of the value of $\varepsilon$ in the domain of ${b}$. Hence, $G$ is a $C^k$-diffeomorphism. Moreover, we have shown that $\beta$ and $\bar{P}$ are $C^k$-differentiable and surjective. Thus, $\beta$ and $\bar{P}$ are $C^k$-bijections with $C^k$-inverses
    \begin{align}
        \bar{P}^{-1} = \beta \circ G^{-1} \quad \text{and} \quad \beta^{-1} = G^{-1} \circ \bar{P}.
    \end{align}
\end{proof}

\begin{proof}[Proof of Theorem \ref{thm3}]
    By Theorem \ref{thm1} and Theorem \ref{thm2}, we consider a periodic branch $\mathcal{B}$ of the extended DDE \eqref{extended-DDE} with time-amplitude parametrization $(x^a_t; r(a))$. In particular, there exists a locally unique vector field on the cyclicity component $\mathcal{O} = \bar{P}\mathcal{B} \subset \mathbb{R}^2$ determined by time differentiation
    \begin{align}
        (x^a(t), x^a(t - 1)) \mapsto (\dot x^a(t), \dot x^a(t - 1)).
    \end{align}
    Next, we derive the specific form \eqref{planar-ODE}. By Theorem \ref{thm1}, $\mathcal{O}$ is a $C^k$-embedded annulus via the global map $G(t, a) = (x^a(t), x^a(t - 1))$ in \eqref{cd-thm1}. In other words, for all $u, v \in \mathcal{O}$ there exist $C^k$-time- and amplitude-maps $\tau(u, v)$, $\alpha(u, v)$ solving 
    \begin{align}\label{def-inv-G}
        (u, v) = \left(x^{\alpha(u,v)}(\tau(u, v)), x^{\alpha(u,v)}(\tau(u, v) - 1)\right).
    \end{align}
    Hence, $\alpha$ given by \eqref{def-inv-G} is precisely the map \eqref{def-alpha}. Moreover, choosing $(u(t), v(t)) = (x^a(t), x^a(t - 1))$ we can always write
    \begin{align}
        \begin{split}
            \dot u(t) &= \dot x^a(t)\\
            &= r(\alpha(u,v))f(u,v).
        \end{split}
    \end{align}
    Analogously, we obtain that
    \begin{align}
        \begin{split}
            \dot v(t) &= \dot x^a(t - 1)\\
            &= \dot x^{\alpha(u,v)}(\tau(u, v) - 1)\\
            &= r(\alpha(u, v))f(x^{\alpha(u,v)}(\tau(u,v) - 1), x^{\alpha(u,v)}(\tau(u,v) - 2)),
        \end{split}
    \end{align}
    therefore, we define
    \begin{align}
        g(u, v) := f(x^{\alpha(u,v)}(\tau(u,v) - 1), x^{\alpha(u,v)}(\tau(u,v) - 2)).
    \end{align}
    Notice that the regularity of $g$ is inherited from $G$. Furthermore, differentiating the identity \eqref{def-inv-G}, we obtain
    \begin{align}\label{eq-u-diff}
        (\partial_u \tau, \partial_u \alpha) = \frac{1}{\det(DG(\tau, \alpha))} (\partial_a x^\alpha(\tau - 1), -\dot x^\alpha(\tau - 1)).
    \end{align}
    Thus, we may replace \eqref{eq-u-diff} into $\partial_1 g$ so that
    \begin{align}
        \begin{split}
            \partial_1 g(u, v) &= (\partial_u \tau \dot x^\alpha(\tau - 1) + \partial_u \alpha \partial_a x^\alpha(\tau - 1))\partial_1 f(x^\alpha(\tau - 1), x^\alpha(\tau - 2)) \\
            &\quad + (\partial_u \tau \dot x^\alpha(\tau - 2) + \partial_u \alpha \partial_a x^\alpha(\tau - 2))\partial_2 f(x^\alpha(\tau - 1), x^\alpha(\tau - 2))\\
            &= -\frac{\det(DG(\tau - 1, \alpha))}{\det(DG(\tau, \alpha))} \partial_2 f(x^\alpha(\tau - 1), x^\alpha(\tau - 2)).
        \end{split}
    \end{align}
    Since $G$ is a diffeomorphism,
    \begin{align}
        \frac{\det(DG(\tau - 1, \alpha))}{\det(DG(\tau, \alpha))} > 0,
    \end{align}
    and we conclude that $\partial_1 g(u, v) \partial_2f(u, v) < 0$, as claimed.
\end{proof}

\section{Proof of Lemma \ref{lemma-annulus-homeomorphism}}\label{Sec5}

In this section, we prove the homeomorphism in Lemma \ref{lemma-annulus-homeomorphism}. Our arguments are a discussion on intersections of continuous families of Jordan curves in $\mathbb{R}^2$. Let $\gamma_\ast$ and $\gamma_\dagger$ be two periodic orbits of the scalar DDE \eqref{reference-dde}. We say that the $C^k$-Jordan curves $P\gamma_\ast, P\gamma_\dagger \subset \mathbb{R}^2$ have a \emph{crossing} if $P\gamma_\ast$ has points both in the inside and the outside of $P\gamma_\dagger$. Converserly, we say that $P\gamma_\ast$ and $P\gamma_\dagger$ have a \emph{tangency} if their intersection is nonempty and they do not have a crossing; see Figure \ref{fig3}. Given a continuous one-parameter family of projected periodic orbits $P\tilde{\gamma}_b$, a crossing is \emph{stable}, that is, if $P\gamma^\ast$ and $P\tilde{\gamma}_{b_\ast}$ have a crossing, then there exists a $\delta > 0$ such that $P\gamma^\ast$ and $P\tilde{\gamma}_{b}$ have a crossing for all $|b - b_\ast| < \delta$. In contrast, tangencies are not stable and can be destroyed by small deformations. We begin the discussion by considering the single case in which the family $P\gamma^\dagger_b$ crosses from the exterior of $P\gamma^\ast$ to the interior by intersecting at a $b = b_\ast$, only.

\begin{figure}
\centering
\includegraphics[width = 0.87\textwidth]{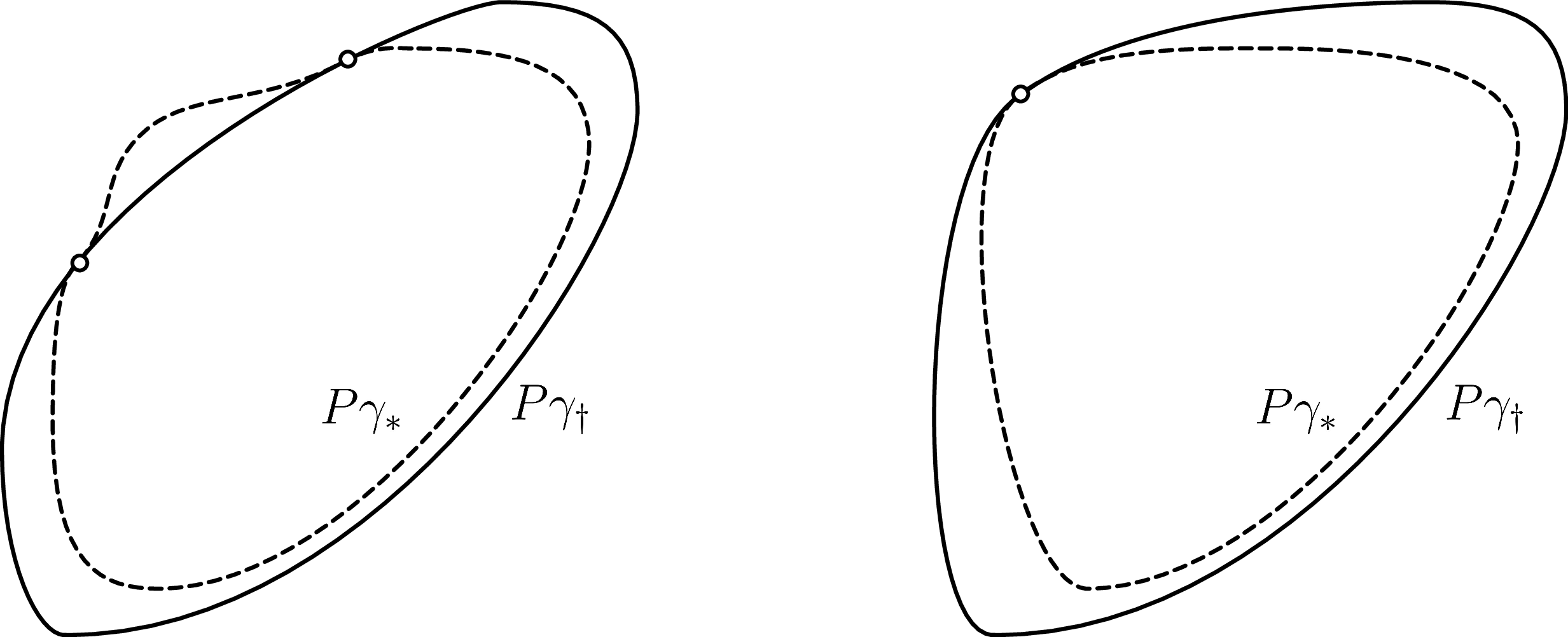}
\caption{Left: \em Crossing with two tangent intersection points (open dots) where the Jordan curve $P\gamma_\ast$ changes the connected components defined by $P\gamma_\dagger$. Right: Tangency, although the curves intersect (open dot), $P\gamma_\ast$ is fully contained in the closure of the interior component of $P\gamma_\dagger$.}
\label{fig3}
\end{figure}

\begin{lemma}\label{lemma-no-degenerate-intersections}
    Let $x^\ast(t)$ and $x^\dagger(t)$ denote two periodic solutions of the DDE \eqref{reference-dde} at delays $r_\ast$ and $r_\dagger$ with minimal periods $p_\ast$ and $p_\dagger$, and orbits $\gamma_\ast$ and $\gamma_\dagger$, respectively. If
    \begin{align}\label{degenerate-intersection}
        P\gamma_\ast = P \gamma_\dagger,
    \end{align}
    then there exists an $m \in \mathbb{Z}$ such that
    \begin{align}\label{degenerate-intersection-consequences}
        x^\ast(t) = x^\dagger((1 + mp_\dagger)t), \quad r_\ast = (1 + mp_\dagger)r_\dagger, \quad \text{and} \quad p_\ast = \frac{p_\dagger}{1 + mp_\dagger}.
    \end{align}
\end{lemma}
\begin{proof}
    Recall that $P\gamma_\ast$ and $P\gamma_\dagger$ are $C^k$-embedded curves in $\mathbb{R} ^2$. Since the images coincide, there exists a $C^k$-function $\tau(t)$ such that
    \begin{align}
        x^\ast(t) &= x^\dagger(\tau(t)),
    \end{align}
    in particular, we have that
    \begin{align}\label{definition-tau}
        \tau(t - p_\ast) = \tau(t) - p_\dagger \quad \text{and} \quad \tau(t - 1) = \tau(t) - 1 - mp_\dagger,
    \end{align}
    for some $m \in \mathbb{Z}$. Differentiating $x^\ast(t)$, we obtain
    \begin{align}
        \dot{x}^\ast (t) = \dot \tau(t) r_\dagger f(x^\ast(t), x^\ast(t - 1)),
    \end{align}
    which shows $\dot \tau(t) = r_\ast/r_\dagger$. Thus, without loss of generality, we can choose
    \begin{align}
        \tau(t) = \frac{r_\ast}{r_\dagger}t.
    \end{align}
    Using \eqref{definition-tau} yields
    \begin{align}
        r_\ast p_\ast = r_\dagger p_\dagger \quad \text{and} \quad r_\ast = (1 + mp_\dagger) r_\dagger.
    \end{align}
    Hence, we obtain the identities \eqref{degenerate-intersection-consequences}.
\end{proof}

\begin{lemma}\label{lemma-constant-product}
    Let $\gamma_\ast$ denote a periodic orbit of the DDE \eqref{reference-dde} at delay $r_\ast \neq 0$ with minimal period $p_\ast$. Let $\tilde{\gamma}_{b}$ denote the orbits of the ${b}$-continuation of any periodic solution of DDE \eqref{reference-dde} as per Lemma \ref{lemma-beta-hyperbolic-continuation} and Lemma \ref{lemma-beta-r-continuation}. If there exists a $\delta > 0$ such that
    \begin{align}\label{nonempty-intersection}
        P\gamma_\ast \cap P\tilde{\gamma}_{b} \neq \emptyset, \quad \text{for all }{b} \in (-\delta, \delta),
    \end{align}
    then $\tilde{p}({b})\tilde{r}({b})$ is constant for all ${b} \in (-\delta, \delta)$ and
    \begin{align}
        \text{either }\tilde{r}'({b}) = \tilde{p}'({b}) = 0,  &\quad \text{for all } |{b}| < \delta \quad \text{or} \quad p_\ast \in \mathbb{R} \setminus \mathbb{Q}. 
    \end{align}
\end{lemma}
\begin{proof}
    Recall that the rescaled functions
    \begin{align}
        x^\ast((1 + mp_\ast)t) \quad \text{and} \quad \tilde{x}^{b}((1 + m\tilde{p}({b}))t),\quad m\in\mathbb{Z},
    \end{align}
    solve the DDE \eqref{reference-dde} for the delays 
    \begin{align}
        r^{(m)}_\ast := (1 + mp_\ast) r_\ast\quad \text{and} \quad\tilde{r}^{(m)}({b}) := (1 + m\tilde{p}({b}))\tilde{r}({b}),
    \end{align}
    respectively. Let us assume that $\tilde{p}({b})\tilde{r}({b})$ is not constant for $|{b}| < \tilde{\delta}$. In particular, the quantity
    \begin{align}
        \mathcal{R}(m) :=\sup \left\{|\tilde{r}^{(m)}({b}_1)) - \tilde{r}^{(m)}({b}_2)| : |{b}_1|, |{b}_2| < \delta\right\}, 
    \end{align}
    grows to infinity with $m$. Therefore, for finite $m_\ast$, the height $\mathcal{R}(m_\ast)$ becomes larger than $p_\ast r_\ast$ and we can find ${b}_\ast \in (-\delta, \delta)$ and $n_\ast \in \mathbb{Z}$ such that
    \begin{align}\label{parameter-value-star}
        r_\ast^{(n_\ast)} = \tilde{r}^{(m_\ast)}({b}_\ast);
    \end{align}
    see Figure \ref{fig4}. In particular, both $x^\ast((1 + n_\ast p_\ast)t)$ and $\tilde{x}((1 + m_\ast \tilde{p}({b}_\ast))t)$ solve the DDE
    \eqref{reference-dde} at the delay \eqref{parameter-value-star}. If we denote the respective orbits by $\gamma^{n_\ast} \subset C$ and $\tilde{\gamma}^{m_\ast} \subset C$, then direct substitution shows that
    \begin{align}
        P\gamma^{n_\ast} = P\gamma_\ast \quad \text{and} \quad P\tilde{\gamma}^{m_\ast} = P\tilde{\gamma}_{{b}_\ast}.
    \end{align}
    Hence $P\gamma^{n_\ast} \cap P\tilde{\gamma}^{m_\ast} \neq \emptyset$, in contradiction to the nesting property \cite[Lemma 5.7]{MPSe962} unless $P\gamma_\ast = P\tilde{\gamma}_{b_\ast}$. Since our argument is valid on any open subset of $(-\delta, \delta)$, we conclude that either $\tilde{p}({b})\tilde{r}({b})$ is constant on $(-\delta, \delta)$ or $P\gamma_\ast = P\tilde{\gamma}_b$ on a dense subset of $(-\delta, \delta)$. By continuity, in either case, we obtain that $\tilde{p}({b})\tilde{r}({b})$ is constant for $|{b}| < \delta$.

    \begin{figure}
    \centering
    \includegraphics[width=0.87\textwidth]{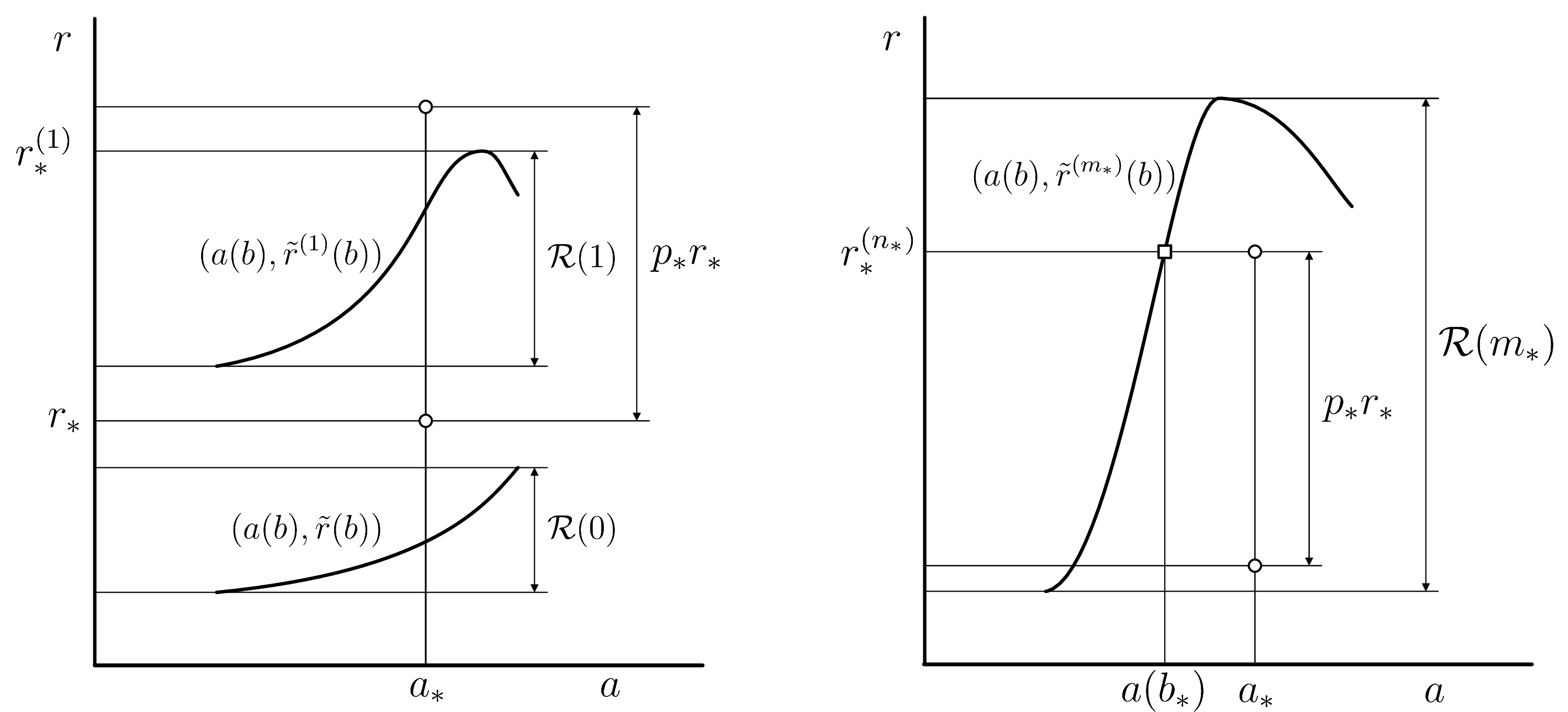}
    \caption{\em Left: Delay maps $\tilde{r}(b)$ associated to periodic branches of the extended DDE \eqref{extended-DDE} that appear by the rescaling symmetry \eqref{def-resc-sym}. Right: If $\tilde{p}(b)\tilde{r}(b)$ is not constant, then the height $\mathcal{R}(m)$ of the branches grows to infinity. Thus, it produces intersections of periodic orbits of the monotone delayed feedback DDE \eqref{reference-dde} at the same delay value $r_\ast^{(n_\ast)}$.}
    \label{fig4}
    \end{figure}
    
    To see the remainder, by contradiction, suppose that $\tilde{p}'(0) \neq 0$ and $p_\ast \in \mathbb{Q}$. Hence, we can find ${b}_\ast$ small such that $\tilde{p}({b}_\ast) \in \mathbb{Q}$ and there exists $M \in \mathbb{N}$ such that
    \begin{align}
        \tilde{x}^\ast(t - M) = \tilde{x}^\ast(t) \quad \text{and} \quad \tilde{x}^{{b}_\ast}(t - M) = \tilde{x}^{{b}_\ast}(t).
    \end{align}
    In particular, the $M$-vectors
    \begin{align}
        u^{\ast}_j(t) := \tilde{x}^{\ast}\left(\frac{t}{r_\ast} - j\right) \quad \text{and} \quad \tilde{u}_j(t) := \tilde{x}^{{b}_\ast}\left(\frac{t}{\tilde{r}({b}_\ast)} - j\right), \quad j = 0, \dots, M - 1,
    \end{align}
    are periodic solutions to the monotone cyclic feedback system
    \begin{align}
        \dot u_j(t) = f(u_j(t), u_{j+1}(t)), \quad j \mod M.
    \end{align}
    Moreover, we have that
    \begin{align}
        P\gamma_\ast = \left\{\left(u^{\ast}_0(t), u^{\ast}_1(t)\right) : t\in \mathbb{R}\right\} \quad \text{and} \quad P\tilde{\gamma}_{{b}_\ast} = \left\{\left(\tilde{u}_0(t), \tilde{u}_1(t)\right) : t\in \mathbb{R}\right\},
    \end{align}
    therefore, $P\gamma_\ast \cap P\tilde{\gamma}_{{b}_\ast} \neq \emptyset$ contradicts the nesting property of monotone cyclic feedback systems \cite[Proposition 3.2]{MPSm90}.
\end{proof}

\begin{proof}[Proof of Lemma \ref{lemma-annulus-homeomorphism}]
    Since $\tilde{G}$ is continuous by construction and we are considering metric spaces, it is sufficient that we show that $\tilde{G}$ is injective.
    
    \textit{Step 1:} If $\tilde{G}(t_1, {b}) = \tilde{G}(t_2, {b})$, then $t_1 - t_2 = m \tilde{p}({b})$ for some $m \in \mathbb{Z}$.\\
    Let $\tilde{\gamma}_{b}$ be the family of orbits parametrized by $\tilde{\beta}$. By construction, we have that $P\tilde{\gamma}_{b} = \{\tilde{G}(t, b) : t\in \mathbb{R}\}$ and recall the Poincar\'{e}--Bendixson theorem for scalar DDEs with monotone feedback \cite[Theorem 2.1]{MPSe962}. Thus, $t \mapsto \tilde{G}(t, {b})$ is a $C^1$-embedding of $\mathbb{R}/\tilde{p}({b})\mathbb{Z}$ with image $P\tilde{\gamma}_{{b}}$. As a result, we obtain that if 
    \begin{align}
        \left(\tilde{x}^{b}(t_1), \tilde{x}^{b}(t_1 - 1)\right) = \left(\tilde{x}^{b}(t_2), \tilde{x}^{b}(t_2 - 1)\right),
    \end{align}
    then $t_1 - t_2 = m\tilde{p}({b})$ for some $m \in \mathbb{Z}$.

    \textit{Step 2:} If $r'({b}) = 0$ for all ${b} \in (-\varepsilon, \varepsilon)$, then Lemma \ref{lemma-annulus-homeomorphism} holds.\\
    Indeed, if $\tilde{\gamma}_{{b}_1}$ and $\tilde{\gamma}_{{b}_2}$ are different orbits of the DDE \eqref{reference-dde} for the same delay $\tilde{r}({b}_1) = \tilde{r}({b}_2)$, then, by \cite[Lemma 5.7]{MPSe961} the planar projections are nested, that is,
    \begin{align}
        {P} \tilde{\gamma}_{{b}_1} \cap {P} \tilde{\gamma}_{{b}_2} = \emptyset.
    \end{align}
    Together with \textit{Step 1}, this proves that if the map $\tilde{r}({b})$ is constant on $(-\varepsilon, \varepsilon)$, then $\tilde{G}$ is injective restricted to the annulus \eqref{domain-equation-annulus}.

    \textit{Step 3:} If ${P} \tilde{\gamma}_{{b}_1}  = {P} \tilde{\gamma}_{{b}_2}$, then ${b}_1 = {b}_2$. \\
    By Lemma \ref{lemma-no-degenerate-intersections}, we have that 
    \begin{align}
        \tilde{p}({b}_1) = \frac{\tilde{p}({b}_2)}{1 + m\tilde{p}({b}_2)},
    \end{align}
    for some $m \in \mathbb{Z}$. However, by the continuity of $\tilde{p}$, there exists an $n\in \mathbb{N}$ such that $\tilde{p}({b}_1), \tilde{p}({b}_2) \in J_n$. Hence, $m = 0$ and $\tilde{\gamma}_{{b}_1}  = \tilde{\gamma}_{{b}_2}$, by uniqueness of the implicit function theorem used in Lemma \ref{lemma-beta-hyperbolic-continuation} and Lemma \ref{lemma-beta-r-continuation}, we conclude that ${b}_1 = {b}_2$.

    \textit{Step 4:} Let ${P} \tilde{\gamma}_{{b}_1} \cap {P} \tilde{\gamma}_{{b}_2} \neq \emptyset$ and $P\tilde{\gamma}_{{b}_1} \neq P\tilde{\gamma}_{{b}_2}$. Then one of the following holds:
    \begin{enumerate}
        \item Either there exist ${b}_1^\ast, {b}_2^\ast \in [{b}_1, {b}_2]$ such that $P\tilde{\gamma}_{{b}_1^\ast} \cap {P} \tilde{\gamma}_{{b}_2^\ast} \neq \emptyset$ is a crossing, or
        \item ${P} \tilde{\gamma}_{{b}_1^\ast} \cap {P} \tilde{\gamma}_{{b}_2^\ast} \neq \emptyset$ for all ${b}_1^\ast, {b}_2^\ast \in [{b}_1, {b}_2]$.
    \end{enumerate}
    We depict the argument in Figure \ref{fig5}. If 1. above does not hold, then $P\tilde{\gamma}_{{b}_1}$ and $P\tilde{\gamma}_{{b}_2}$ intersect at a tangency. If we assume without loss of generality that $P\tilde{\gamma}_{{b}_2}$ lies on the closure of the outside of $P\tilde{\gamma}_{{b}_1}$, then $P\tilde{\gamma}_{b}$ is contained in the closure of the outside of $P\tilde{\gamma}_{{b}_1}$ for all ${b} \in [{b}_1, {b}_2]$. Otherwise, $P\tilde{\gamma}_{b}$ lies in the inside of $P\tilde{\gamma}_{{b}_1}$ and must intersect $P\tilde{\gamma}_{{b}_1}$ to pass to the the outside as ${b} \to {b}_2$. By \textit{Step 3}, such intersection must be a crossing.

    An analogous argument shows that $P\tilde{\gamma}_{b}$ lies inside the closure of the interior of $P\tilde{\gamma}_{{b}_2}$ for all ${b} \in [{b}_1, {b}_2]$. Hence, we conclude that 
    \begin{align}
        \emptyset \neq {P} \tilde{\gamma}_{{b}_1} \cap {P} \tilde{\gamma}_{{b}_2} \subset {P} \tilde{\gamma}_{{b}_1^\ast} \cap {P} \tilde{\gamma}_{{b}_2^\ast}, \quad \text{for all }{b}_1^\ast,{b}_2^\ast \in [{b}_1, {b}_2],
    \end{align}
    which implies 2.

    \begin{figure}
    \centering
    \includegraphics[width=0.87\textwidth]{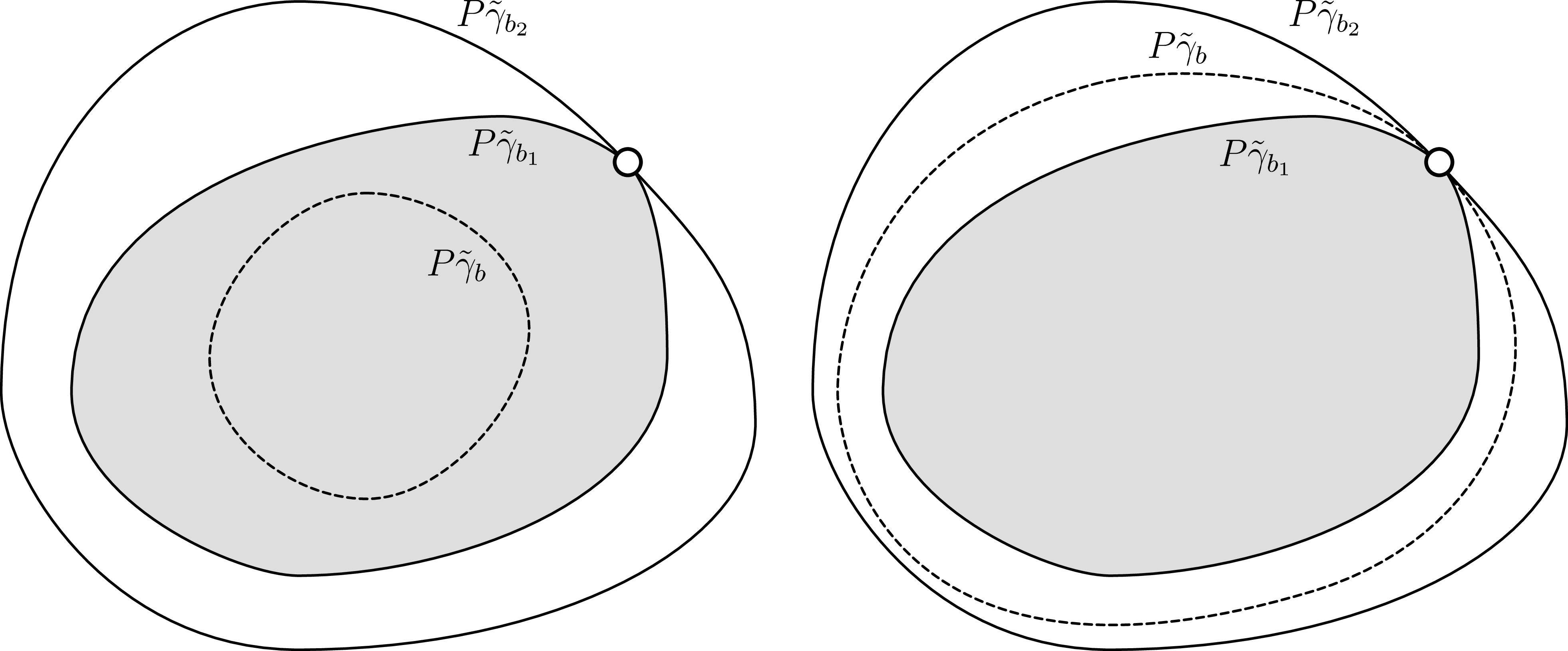}
    \caption{\em Possible relative configurations of the projections $P\tilde{\gamma}_b$ if $P\tilde{\gamma}_{b_1}$ and $P\tilde{\gamma}_{b_2}$ have a tangency (open dot). Left: If $P\tilde{\gamma}_b$ lies in the interior of $P\tilde{\gamma}_{b_1}$, then deforming it into $P\tilde{\gamma}_{b_2}$ yields a crossing. The same happens if $P\tilde{\gamma}_{b}$ lies in the outside of $P\tilde{\gamma}_{b_2}$ instead. Right: Alternative case, $P\tilde{\gamma}_b$ is pinned tangentially in between $P\tilde{\gamma}_{b_1}$ and $P\tilde{\gamma}_{b_2}$ for all values of $b$.}
    \label{fig5}
    \end{figure}

    \textit{Step 5:} Lemma \ref{lemma-annulus-homeomorphism} holds if $\tilde{r}'(0) \neq 0$.\\
    If $\tilde{r}'(0) \neq 0$, then there exists an $\varepsilon > 0$ small such that $\tilde{r}'({b}) \neq 0$ for all ${b} \in (-\varepsilon, \varepsilon)$. We proceed by contradiction and suppose that there exist delays ${b}_1 < {b}_2$ such that
    \begin{align} \label{intersection-projections}
        {P} \tilde{\gamma}_{{b}_1} \cap {P} \tilde{\gamma}_{{b}_2} \neq \emptyset.
    \end{align}
    If case 1. in \textit{Step 4} holds, then, by the stability of crossings, we can find a $\delta > 0$ such that $P \tilde{\gamma}_{{b}_1^\ast} \cap P\tilde{\gamma}_{{b}} \neq \emptyset$ for all ${b} \in ({b}_2^\ast - \delta, {b}_2^\ast + \delta)$. Applying Lemma \ref{lemma-constant-product}, and recalling that $\tilde{r}'({b}_2^\ast), \neq 0$ we obtain that $\tilde{p}({b}_1^\ast)$ is irrational. Naturally, the argument can be repeated to show that $\tilde{p}({b})$ is irrational, and hence constant, near ${b}_1^\ast$. For the same reason, but exchanging the indices of the periodic orbits, we conclude that $\tilde{p}({b})\tilde{r}({b})$ is constant close to ${b}_1^\ast$. Hence $\tilde{r}'({b}) = 0$ in a neighborhood of ${b}_1^\ast$, in contradiction to $\tilde{r}'(0) \neq 0$.
    
    If case 2. in \textit{Step 4} holds, then ${P} \tilde{\gamma}_{{b}_1^\ast} \cap {P} \tilde{\gamma}_{{b}_2^\ast} \neq \emptyset$ for all ${b}_1^\ast, {b}_2^\ast \in [{b}_1, {b}_2]$. We apply Lemma \ref{lemma-constant-product} with $\gamma_\ast = \tilde{\gamma}_{b}$ for all ${b} \in [{b}_1, {b}_2]$, which shows that $\tilde{p}'({b}) = \tilde{r}'({b}) = 0$ for all ${b} \in [{b}_1, {b}_2]$, in contradiction to $\tilde{r}'(0) \neq 0$.
    
    \textit{Step 6:} Lemma \ref{lemma-annulus-homeomorphism} holds if $\tilde{r}'(0) = 0$.
    
    Recall from Lemma \ref{lemma-beta-r-continuation} that $\partial_{b}\tilde{x}^0(0) \neq 0$. By construction, $\tilde{x}^{b}(0)$ is the amplitude of the corresponding periodic solution, and we may choose $a = \tilde{x}^{b}(0)$ as a coordinate. In amplitude coordinates, we denote $\gamma_a := \tilde{\gamma}_{{b}(a)}$, $r(a) := \tilde{r}(b(a))$, and $G(t, a) := \tilde{G}(t, {b}(a))$, and assume without loss of generality that ${b}'(a) > 0$. By contradiction, suppose that there exist $a_1 < a_2$ such that
    \begin{align}
        P\gamma_{a_1} \cap P\gamma_{a_2} \neq \emptyset.
    \end{align}
    By \textit{Step 2}, if $r(a)$ is constant on $(a_1, a_2)$, we are done. Otherwise, we claim that there exists an $a_\ast \in (a_1, a_2)$ such that
    \begin{align}\label{step6-claim}
        \text{ $r'(a_\ast) \neq 0$ and $P\gamma_{a_\ast}$ has a crossing intersection with either $P\gamma_{a_1}$ or $P\gamma_{a_2}$.}
    \end{align}
    Ideed, we choose an $\tilde{a}_\ast \in (a_1, a_2)$ such that $r'(\tilde{a}_\ast) \neq 0$. Since the amplitudes are achieved on the nullcline line $f^{-1}(0) \subset \mathbb{R}^2$, all curves $P\gamma_a$ are parallel at the nullcline. Thus, the triangle determined by $f^{-1}(0) \subset \mathbb{R}^2$, $P\gamma_{a_1}$, and $P\gamma_{a_2}$ can only be escaped through $P\gamma_{a_1}$, and $P\gamma_{a_2}$; see Figure \ref{fig6}. If $P\gamma_{\tilde{a}_\ast}$ crosses either $P\gamma_{a_1}$ or $P\gamma_{a_2}$, then the claim is true with $a_\ast := \tilde{a}_\ast$. If $P\gamma_{\tilde{a}_\ast}$ crosses neither $P\gamma_{a_1}$ nor $P\gamma_{a_2}$, we claim that there exists a $\delta > 0$ such that the result holds with $a_\ast = \tilde{a}_\ast + \delta$. Indeed, this situation only happens if $P\gamma_{a_\ast}$ leaves the triangle through a tangency at the intersection of $P\gamma_{a_1}$ and $P\gamma_{a_2}$; see Figure \eqref{fig6}. Thus, we obtain 
    \begin{align}
        \emptyset \neq P \gamma_{a_1} \cap P\gamma_{a_2} \subset P\gamma_{a_1^\ast} \cap P\gamma_{a_2^\ast},\quad \text{for all }a_1^\ast,a_2^\ast \in(\tilde{a}_\ast - \delta, \tilde{a}_\ast + \delta),
    \end{align}
    and Lemma \ref{lemma-constant-product} shows that $r'(\tilde{a}_\ast) = 0$ in contradiction to our assumptions. This proves the claim \eqref{step6-claim}.

    \begin{figure}[h]
        \centering
        \includegraphics[width=0.87\textwidth]{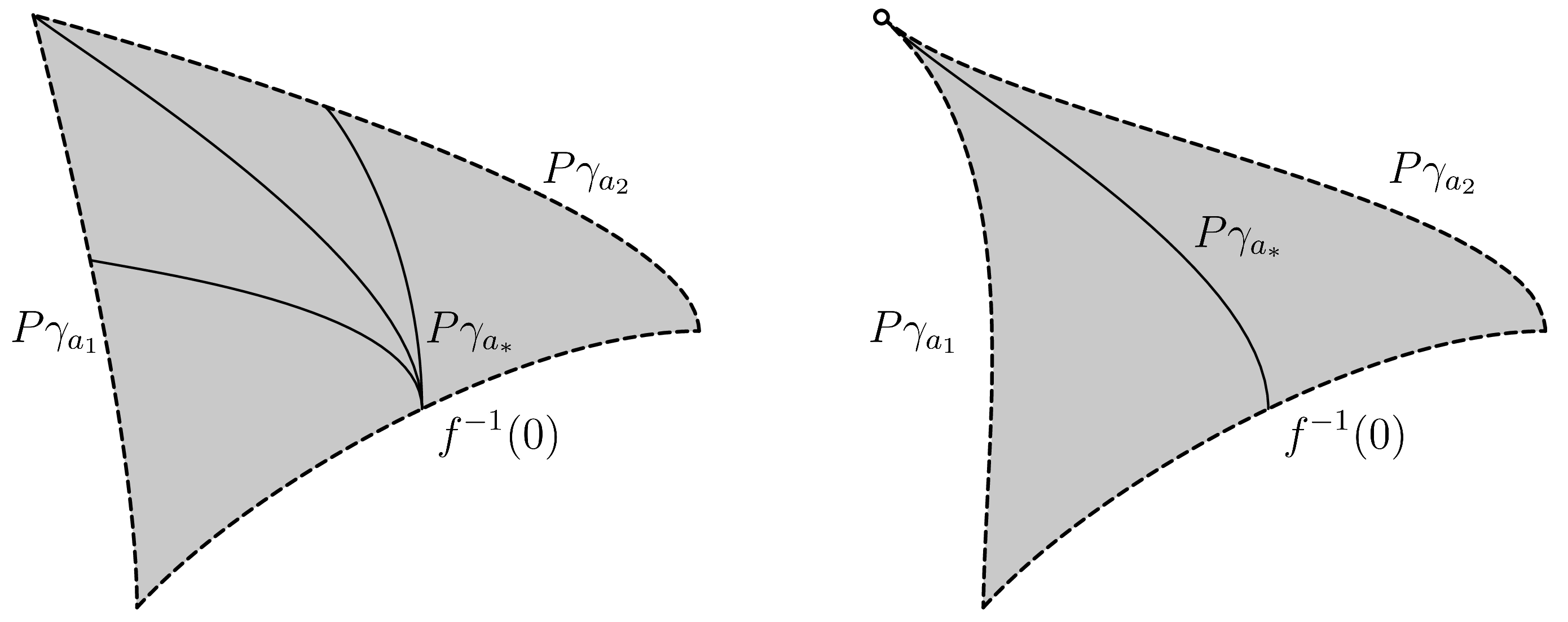}
        \caption{\em Triangles formed by the projections $P\gamma_{a_i}$ and the nullcline $f^{-1}(0)$. Left: If the projected orbits cross, then $P\gamma_{a_\ast}$ must cross one of them. Right: If $P\gamma_{a_\ast}$ crosses neither $P\gamma_{a_1}$ nor $P\gamma_{a_2}$, then it leaves the triangle through a tangency at the top tip (open dot).}
        \label{fig6}
    \end{figure}

    Next, for simplicity, we assume that $P\gamma_{a_\ast}$ in claim \eqref{step6-claim} crosses $P\gamma_{a_1}$. The case when $P\gamma_{a_\ast}$ crosses $P\gamma_{a_2}$ can be treated similarly. Since $a_1 < a_\ast$, the intersection of $P\gamma_{a_\ast}$ with the nullcline $f^{-1}(0)$ lies outside of $P\gamma_a$ for all $a \in [a_1, a_\ast)$. Recalling that $P\gamma_{a_\ast}$ crosses $P\gamma_{a_1}$, we define the interval 
    \begin{align}
        \mathcal{I} := \left\{a \in (a_1, a_\ast) : P\gamma_{\tilde{a}} \text{ crosses } P\gamma_{a_\ast} \text{ for all } \tilde{a} \in (a_1, a) \right\}.
    \end{align}
    By construction, $\mathcal{I}$ is open in $(a_1, a_\ast)$, nonempty, and connected. We claim that $\mathcal{I}$ is also closed. Indeed, suppose that $\overline{a}_1 = \sup \mathcal{I} < a_\ast$, then Lemma \ref{lemma-constant-product} implies that $r'(a) = 0$ for all $a\in\mathcal{I}$ and \textit{Step 2} shows that $P\gamma_a$ lies outside of $P\gamma_{a_1}$ for all $a \in \mathcal{I}$. In particular, $P\gamma_{\overline{a}_1}$ contains $P\gamma_{a_1}$ in its inside. Moreover, since $a_1 < a_\ast$, the intersection of $P\gamma_{a_\ast}$ with the nullcline $f^{-1}(0)$ lies outside of $P\gamma_a$ for all $a \in \mathcal{I}$. Recalling that $P\gamma_{a_\ast}$ crosses $P\gamma_{a_1}$, we conclude that $P\gamma_{a_\ast}$ crosses $P\gamma_{\overline{a}_1}$. Hence $\mathcal{I} = (a_1, a_\ast)$, and, by Lemma \ref{lemma-constant-product}, $r'(a) = 0$ for all $a \in (a_1, a_\ast)$. By the continuity of $\tilde{r}(a)$, we reach a contradiction to $r'(a_\ast) \neq 0$, which finishes the proof.
\end{proof}

\section{Proof of Theorem \ref{thm4}}\label{Sec6}
\begin{proof}
    The proof is direct. Assume that two periodic branches $\mathcal{B}$ and $\hat{\mathcal{B}}$ are such that their cyclicity components $\mathcal{O}$ and $\hat{\mathcal{O}}$ intersect. Then we prove that both $\mathcal{B}$ and $\hat{\mathcal{B}}$ emanate from a Hopf bifurcation point and show Theorem \ref{thm4}. Indeed, in time-amplitude coordinates, we have that
    \begin{align}\label{eq-fam-sols}
        \dot x^a(t) = r(a)f(x^a(t), x^a(t - 1)) \quad \text{and} \quad \dot {\hat{x}}^a(t) = \hat{r}(a)f(\hat{x}^a(t), \hat{x}^a(t - 1)).
    \end{align}
    Recalling that the amplitude are $\alpha(\mathcal{O})$ and $\hat{\alpha}(\hat{\mathcal{O}})$, we consider
    \begin{align}
        \underline{a} := \inf(\alpha(\mathcal{O})\cap \hat{\alpha}(\hat{\mathcal{O}})).
    \end{align}
    We claim that 
    \begin{align}\label{eq-claim-equil}
        f(\underline{a}, \underline{a}) = 0.
    \end{align}
    Indeed, by the time rescaling symmetry \eqref{def-resc-sym}, we may assume that $p, \hat{p} \in J_n$ are uniformly bounded. Next, we rescale time via
    \begin{align}\label{def-norm-fam}
        \boldsymbol{x}^a(t) := x^a(p(a)t) \quad \text{and} \quad \hat{\boldsymbol{x}}^a(t) := \hat{x}^a(\hat{p}(a)t),
    \end{align}
    by construction, the normalized solutions \eqref{def-norm-fam} have period $1$ and satisfy the DDEs
    \begin{align}\label{normalized-DDEa}
        \dot{\boldsymbol{x}}^a(t) &= p(a)r(a) f\left(\boldsymbol{x}^a(t), \boldsymbol{x}^a\left(t - \frac{1}{p(a)}\right)\right),\\
        \dot{\hat{\boldsymbol{x}}}^a(t) &= \hat{p}(a)\hat{r}(a) f\left(\hat{\boldsymbol{x}}^a(t), \hat{\boldsymbol{x}}^a\left(t - \frac{1}{\hat{p}(a)}\right)\right).\label{normalized-DDEb}
    \end{align}
    By Lemma \ref{lemma-constant-product}, we have that $p(a)r(a)$ and $\hat{p}(a)\hat{r}(a)$ are constant for all $a \in \alpha(\mathcal{O})\cap \hat{\alpha}(\hat{\mathcal{O}})$ sufficiently close to $\underline{a}$. Hence, any two sequences of normalized periodic solutions
    \begin{align}\label{def-per-seq}
        \left\{\boldsymbol{x}^{a_n}(t)\right\}_{n \in \mathbb{N}} \quad \text{and} \quad \left\{\hat{\boldsymbol{x}}^{a_n}(t)\right\}_{n \in \mathbb{N}},
    \end{align}
    with $a_n \to \underline{a}$ are uniformly bounded as elements of $C^2(\mathbb{R}/\mathbb{Z}, \mathbb{R})$. By the Arzel\'{a}--Ascoli theorem, taking subsequences if necessary, there exist functions $\boldsymbol{x}^{\underline{a}}(t), \, \boldsymbol{x}^{\underline{a}}(t) \in C^1(\mathbb{R}/\mathbb{Z}, \mathbb{R})$ that solve the normalized DDEs \eqref{normalized-DDEa}--\eqref{normalized-DDEb} for finite values $p(\underline{a}),\, \hat{p}(\underline{a}) > 0$. Hence, we have constructed periodic solutions $x^{\underline{a}}(t)$ and $\hat{x}^{\underline{a}}(t)$ with equal amplitude $\underline{a}$ of the DDE families \eqref{eq-fam-sols}.
    
    We claim that we have reached a contradiction unless
    \begin{align}\label{eq-equil-ident}
        x^{\underline{a}}(t) \equiv \hat{x}^{\underline{a}}(t) \equiv \underline{a}.
    \end{align}
    Indeed, if $x^{\underline{a}}(t)$ is not constant, then neither is $\hat{x}^{\underline{a}}(t)$. Otherwise, the constant solution $\underline{a}$ solves the DDE \eqref{reference-dde} for all values of $r \in \mathbb{R}$ and the existence of a nonconstant periodic solution $x^{\underline{a}}(t)$ contradicts the Poincar\'{e}--Bendixson theorem \cite[Theorem 2.1]{MPSe962}. Hence $(x^{\underline{a}}_t; r(\underline{a})) \in \mathcal{B}$ and $(\hat{x}^{\underline{a}}_t; \hat{r}(\underline{a})) \in \hat{\mathcal{B}}$ admit a local continuation for smaller values of the amplitude $\underline{a}$. However, this is a contradiction to $\underline{a}$ being the infimum of the intersection of the amplitude ranges. We conclude that the identity \eqref{eq-equil-ident} holds and, therefore, the claim \eqref{eq-claim-equil} follows.
    
    Notice that, by the convergence of the sequences \eqref{def-per-seq} as $a_n \to \underline{a}$, we conclude that the map
    \begin{align}
        \mathrm{Id} - \partial_2 S\left(p(\underline{a}), x^{\underline{a}}_0; r(\underline{a})\right) : C \longrightarrow C,
    \end{align}
    is not invertible. Otherwise, by the implicit function theorem, the constant function $\underline{a} \in C$ is a locally unique zero of $\mathrm{Id} - S(p(a), \underline{a}; r(a))$ for all values of $a$. Thus, contradicting that $(\underline{a}; \hat{r}(\underline{a})$ is an accumulation point of periodic orbits.
    
    Following \cite{HaLu93}, we conclude that the characteristic equation
    \begin{align}\label{eq-char}
        i\nu = r(\underline{a})\left(\partial_1 f(\underline{a}, \underline{a}) + \mathrm{e}^{-i\nu}\partial_2 f(\underline{a},\underline{a})\right), 
    \end{align}
    possesses solutions $\nu \in \mathbb{R}$. This is only possible if:
    \begin{itemize}
        \item Either $\partial_1 f(\underline{a}, \underline{a}) = -\partial_2 f(\underline{a},\underline{a})$ and $\nu = 0$ is a solution with multiplicity two of the characteristic equation \eqref{eq-char}, or
        \item $(\underline{a}; r(\underline{a}))$ is a Hopf point, that is, $\nu = 2\pi / p(\underline{a})$ is a simple solution of \eqref{eq-char}.
    \end{itemize}
    We reduce our analysis to the Hopf bifurcation scenario by perturbing the nonlinearity $f$ into
    \begin{align}
        \tilde{f}(u, v) = f(u, v) + \varepsilon \iota(u, v) (v - \underline{a}),
    \end{align}
    where $\iota$ is a cut-off function with support contained in an arbitrarily small ball close to $(\underline{a}, \underline{a})$ and such that $\iota(\underline{a}, \underline{a}) = 1$ in a ball around $(\underline{a}, \underline{a})$. Thus, we may choose $|\varepsilon|$ small enough such that our previous analysis holds, but $\partial_1 \tilde{f}(\underline{a}, \underline{a}) \neq - \partial_2 \tilde{f}(\underline{a}, \underline{a})$. Arguing analogously, we conclude that $(\underline{a}; \hat{r}(\underline{a}))$ is also a Hopf point of the extended DDE \eqref{extended-DDE}. By direct examination of the characteristic equations
    \begin{align}
        \nu = r(\underline{a})\left(\partial_1 f(\underline{a}, \underline{a}) + \mathrm{e}^{-\nu}\partial_2 f(\underline{a},\underline{a})\right),  \quad \text{and} \quad \hat{\nu} = \hat{r}(\underline{a})\left(\partial_1 f(\underline{a}, \underline{a}) + \mathrm{e}^{-\hat{\nu}}\partial_2 f(\underline{a},\underline{a})\right), 
    \end{align}
    we obtain that
    \begin{align}
        \hat{r}(\underline{a}) = (1 + mp(\underline{a})) r(\underline{a}), \quad \text{for some } m\in \mathbb{Z}.
    \end{align}
    In turn, the uniqueness of the Hopf branch completes the proof.
\end{proof}

\section{Proof of Lemma \ref{lemma-projections}} \label{Sec7}
To define the projections in Lemma \ref{lemma-projections}, we use the so-called \emph{formal adjoint equation}; see \cite{Henry70, Ha77}. In the following, we use the transposition sign ``$^\mathsf{T}$'' to distinguish functions in the space
\begin{align}
    C^\mathsf{T} := C^0([0, 1], \mathbb{R}).
\end{align}
Then, the \emph{formal adjoint equation of} \eqref{linearized-equation-periodic-orbit} is the linear DDE
\begin{align}\label{formal-adjoint-equation}
    \begin{split}
        \dot y^\mathsf{T}(t) &= -A(t)y^\mathsf{T}(t) - B(t + 1)y^\mathsf{T}(t + 1),\\
        y^\mathsf{T}_0(\vartheta) & = \psi^\mathsf{T}(\vartheta),\text{ for all }\vartheta \in [0,1],
    \end{split}
\end{align}
where the coefficients are given by \eqref{linear-coefficients}. Notice that for any initial data $\psi^\mathsf{T} \in C^\mathsf{T}$, we can solve the DDE \eqref{formal-adjoint-equation} in backwards time direction. To be coherent with \cite{Ha77}, given a solution $y^\mathsf{T}(t)$, $t\leq 1$ of \eqref{formal-adjoint-equation}, the subindex notation in combination with the transpose denotes $y^\mathsf{T}_t(\vartheta) := y^\mathsf{T}(t + \vartheta)$ for $\vartheta \in [0, 1]$. Then, the \emph{formal adjoint monodromy operator} is $L^\mathsf{T}$ given by the relation
\begin{equation}
    L^\mathsf{T}y^\mathsf{T}_0 := y^\mathsf{T}_{-{p_\ast}}.
\end{equation}
We pair $C^\mathsf{T}$ with $C$ via the time-dependent bilinear form
\begin{align}\label{bilinear-form}
    \left[\varphi^\mathsf{T},\varphi\right]_t := \varphi^\mathsf{T}(0)\varphi(0) + \int_0^1 \varphi^\mathsf{T}(\vartheta)B(t+\vartheta)\varphi(\vartheta-1)\, \mathrm{d} \vartheta.
\end{align}
Notice that, by direct differentiation, $[y^\mathsf{T}_t, y_t]_t$ is constant in $t$ along solutions of the formal adjoint pair
\begin{align}
    \begin{split}
        \dot y(t) &= A(t) y(t) + B(t) y(t-1),\\
        \dot y^\mathsf{T}(t) &= -A(t) y^\mathsf{T}(t) + B(t+1) y^\mathsf{T}(t+1).
    \end{split}
\end{align}
More generally, we have the following proposition.

\begin{proposition}\label{propsition-derivative-bilinear-form}
    Consider the nonhomogenous linear DDE
    \begin{align}
    \begin{split}\label{equation-linear-forced}
        \dot y(t) = A(t) y(t) + B(t) y(t-1) + \ell(t),
    \end{split}
    \end{align}
    where $\ell:\mathbb{R} \to \mathbb{R}$ and let $y^\mathsf{T}(t)$ be a solution of the formal adjoint equation \eqref{formal-adjoint-equation}. Then, along a solution $y(t)$ of \eqref{equation-linear-forced}, the bilinear form \eqref{bilinear-form} satisfies the identity
    \begin{align}
        \left[y^\mathsf{T}_t, y_t\right]_t = \left[y^\mathsf{T}_0, y_0\right]_0 + \int_0^ty^\mathsf{T}(t)\ell(t)\, \mathrm{d} t.
    \end{align}
\end{proposition}
\begin{proof}
    The proof follows from direct differentiation of \eqref{bilinear-form} since
    \begin{align}
        \frac{\, \mathrm{d} }{\, \mathrm{d} t}\left[y^\mathsf{T}_t, y_t\right]_t = y^\mathsf{T}(t)\ell(t).
    \end{align}
\end{proof}

\begin{lemma}\label{lemma-relative-zero-position}
    Let ${x}^\ast(t)$ be a normalized periodic solution of the DDE \eqref{reference-dde} with minimal period ${p_\ast} \in J_n$, where $J_n$ are defined in \eqref{In-definition}, and time of depth $q \in (0, {p_\ast})$. Then, ${n_1} := \left\lfloor (n - 1)/ 2\right\rfloor$ is the only integer such that
    \begin{align}\label{N-property}
        {n_1} {p_\ast} < 1 < ({n_1} + 1){p_\ast}.
    \end{align}
    Moreover, $q$ satisfies
    \begin{align}\label{zero-distance-b}
        {n_1} {p_\ast} < 1 < q + {n_1} {p_\ast} < q + 1 < ({n_1} + 1) {p_\ast},& \quad \text{ for all $n$ odd and}\\
        \label{zero-distance-c}
        {n_1} {p_\ast} < q - {p_\ast} + 1 <  q + {n_1} {p_\ast} < 1 < ({n_1} + 1) {p_\ast}, &\quad \text{ for all $n$ even.}
    \end{align}
\end{lemma}
\begin{proof}
    Indeed, the property \eqref{N-property} follows from \eqref{In-definition} since
    \begin{align}
        \frac{n - 1}{2} {p_\ast} < 1 < \frac{n}{2} {p_\ast}.
    \end{align}    
    Next, we show \eqref{zero-distance-b}--\eqref{zero-distance-c}. If we assume ${n_1} = 0$, then the result follows from the so-called zero number \cite{MPSe961}. More precisely, if ${p_\ast} \in J_1$, then the distance between any two zeros of $\dot {x}^\ast(t)$ is bigger than one. Since $\dot {x}^\ast(0) = \dot {x}^\ast(q) = 0$, we obtain \eqref{zero-distance-b}. In case ${p_\ast} \in J_2$, then $\dot {x}^\ast(t)$ possesses at least one zero over any interval of length one and at most two. Hence,
    \begin{align}
        0 < q < 1 < {p_\ast} < q + 1 < {p_\ast} + q,
    \end{align}
    and subtracting ${p_\ast}$ from the final inequalities, we obtain \eqref{zero-distance-c}. To see the general case, notice that if we consider $J_n$ with $n$ odd, then
    \begin{equation}
        \begin{array}{rcl}
            J_n & \longrightarrow & J_1 \\
             p & \longmapsto & \displaystyle \frac{p}{1 - {n_1} p},
        \end{array}
    \end{equation}
    is a bijection. Hence, the time rescaled function 
    \begin{align}
        {x}^\ast((1 - {n_1} {p_\ast})t)
    \end{align}
    solves \eqref{reference-dde} at delay $(1 - {n_1} {p_\ast})r_\ast$ and has minimal period 
    \begin{align}
        \frac{{p_\ast}}{1 - {n_1} {p_\ast}} \in J_1.
    \end{align}
    Since \eqref{zero-distance-b} holds for $J_1$, in general we obtain the inequalities
    \begin{align}
        0 < 1 < \frac{q}{1 - {n_1} {p_\ast}} < \frac{q}{1 - {n_1} {p_\ast}} + 1 < \frac{{p_\ast}}{1 - {n_1} {p_\ast}},
    \end{align}
    hence,
    \begin{align}
        0 < 1 - {n_1} {p_\ast}  < q < q + 1 - {n_1} {p_\ast} < {p_\ast}
    \end{align}
    and, adding ${n_1} {p_\ast}$, we obtain \eqref{zero-distance-b}. If $n$ is even, then the bijection is
    \begin{equation}
        \begin{array}{rcl}
            J_n & \longrightarrow & J_2 \\
             p & \longmapsto & \displaystyle \frac{p}{1 - {n_1} p},
        \end{array}
    \end{equation}
    after rescaling, we obtain
    \begin{align}
        0 < \frac{q - {p_\ast}}{1 - {n_1} {p_\ast}} + 1 < \frac{q}{1 - {n_1} {p_\ast}} < 1 < \frac{{p_\ast}}{1 - {n_1} {p_\ast}},
    \end{align}
    which yields \eqref{zero-distance-c}.
\end{proof}

\begin{remark}
    Notice that the proof of Lemma \ref{lemma-relative-zero-position} shows we can always choose a representative branch where the minimal period of the period solutions satisfies $p_\ast \in J_1$, that is, $p_\ast > 2$. That choice is the slow branch discussed in Section \ref{Sec2}.
\end{remark}
\begin{lemma}\label{lemma-adjoint-eigenfunctions}
    Let ${x}^\ast(t)$ be a periodic solution of the DDE \eqref{reference-dde}. Then the spectrum of $L^\mathsf{T}$ coincides with that of the monodromy operator $L$ solving \eqref{linearized-equation-periodic-orbit}. Moreoever, given the critical eigenvalue $\mu_\mathrm{c}$ and $E^\mathrm{c}$ in Proposition \ref{proposition-spectrum}, for all $\Psi_0 \in E^\mathrm{c}$ such that $E^\mathrm{c} = \mathrm{span}_{\mathbb{R}}\{\dot {x}^\ast_0, \Psi_0\}$, we can find a unique eigenfunction $\Psi^\mathsf{T}_0 \in C^\mathsf{T}$ of the formal adjoint equation \eqref{formal-adjoint-equation} such that
    \begin{align}\label{pure-eigenfunction}
        L^\mathsf{T} \Psi_0^\mathsf{T} = \mu_\mathrm{c} \Psi_0^\mathsf{T}.
    \end{align}
    The maps
    \begin{align}\label{equation-formal-adjoint-eigenfunctions}
        {P}_\Psi \varphi:= \left[\Psi^\mathsf{T}_0, \varphi\right]_0\Psi_0 \quad \text{and} \quad {Q}_\Psi\varphi := \varphi - {P}_\Psi\varphi, \quad \varphi\in C,
    \end{align}
    are projections with range
    \begin{align}\label{identities-range}
        \operatorname{ran} P_\Psi = \mathrm{span}_{\mathbb{R}} \{\Psi_0\} \quad \text{and} \quad \operatorname{ran} Q_\Psi = \mathrm{span}_{\mathbb{R}}\{\dot x_0^\ast\} \oplus R^\mathrm{c}.
    \end{align}
    Moreover, we have that $P_\Psi \dot x_0^\ast = Q_\Psi \Psi_0 = 0$. Analogously to $\dot {x}^\ast(t)$, $\Psi^\mathsf{T}(t)$ possesses two zeros in the interval $[0, {p_\ast})$ and both of them are simple.
\end{lemma}
\begin{proof}
    It is well-known that the eigenfunctions of the formal adjoint equation \eqref{formal-adjoint-equation} can be used to represent projections onto eigenspaces as in the identities \eqref{equation-formal-adjoint-eigenfunctions}. The core idea follows \cite{Henry70}, which shows that, if we consider $\Bar{L}^\mathsf{T}$, then the extension of $L^\mathsf{T}$ to the space of functions of bounded variation on the unit interval, then there exists a representation of the dual $C^\ast$ such that $\bar{L}^\mathsf{T}$ is the adjoint operator of $L$. Since the eigenfunctions of $\bar{L}^\mathsf{T}$ belong to $C^\mathsf{T}$, the spectra of $L^\mathsf{T}$ and $L$ coincide and the generalized eigenspaces have the same dimension. Hence, we follow \cite[Section 4]{Henry70} to find the unique generalized eigenfunction $\Psi^\mathsf{T}(t)$ associated to the critical eigenvalue $\mu_\mathrm{c}$ such that the identities \eqref{equation-formal-adjoint-eigenfunctions}--\eqref{identities-range} hold, $P_\Psi \dot x_0^\ast = 0$, and $Q_\Psi \Psi_0 = 0$. 
    
    However, it is not clear that $\Psi^\mathsf{T}(t)$ satisfies \eqref{pure-eigenfunction}. If $\mu_\mathrm{c} \neq 1$, then the geometric multiplicity of $\mu_\mathrm{c}$ is one and \eqref{pure-eigenfunction} holds. If $\mu_\mathrm{c} = 1$, then we may choose a second generalized eigenfunction $\xi^\mathsf{T}(t)$ of $L^\mathsf{T}$ such that $\ker(L^\mathsf{T} - \mathrm{Id})^2 = \operatorname{span}_\mathbb{R}\{\Psi^\mathsf{T}_0, \xi^\mathsf{T}_0\}$. By \cite{Henry70}, we can pick $\xi^\mathsf{T}_0$ in such a way that
    \begin{align}
        \left[\xi^\mathsf{T}_0, \dot {x}^\ast_0\right]_0 = 1, \quad \left[\xi^\mathsf{T}_0, \Psi_0\right]_0 = 0, \quad \text{and} \quad   \left[\xi^\mathsf{T}_0, \varphi\right]_0 = 0, \; \text{for all } \varphi \in  R^\mathrm{c}.
    \end{align}
    On the one hand, by Floquet theory \cite{HaLu93}, there exists a $\kappa^\mathsf{T} \in \mathbb{R}$ such that 
    \begin{align}\label{generalized-eigenfunction-Psi}
        \Psi^\mathsf{T}(t + {p_\ast}) = \Psi(t)^\mathsf{T} + \kappa^\mathsf{T} \xi^\mathsf{T}(t),
    \end{align}
    on the other hand, by Proposition \ref{propsition-derivative-bilinear-form}, we have that
    \begin{align}
        \begin{split}
            \left[\Psi^\mathsf{T}_{p_\ast}, \dot {x}^\ast_{p_\ast}\right]_{p_\ast} &= \left[\Psi^\mathsf{T}_0, \dot {x}^\ast_0\right]_0 + \kappa^\mathsf{T}\left[\xi^\mathsf{T}_0, \dot {x}^\ast_0\right]_0\\
            &= \kappa^\mathsf{T}\\
            &= 0.
        \end{split}
    \end{align}
    Hence \eqref{generalized-eigenfunction-Psi} implies \eqref{pure-eigenfunction}.
    
    Finally, we prove the claims on the number of zeros. The formal adjoint equation \eqref{formal-adjoint-equation} meets the assumptions of \cite[Theorem 2.2]{MPSe961}. Hence all zeros of $\Psi^\mathsf{T}(t)$ are simple. By contradiction, suppose that $\Psi^\mathsf{T}(t)$, has a different number of zeros than $\dot {x}^\ast(t)$ for $t \in (0, {p_\ast}]$. On the one hand, the number of zeros of $\Psi^\mathsf{T}(t)$ and $\dot {x}^\ast(t)$ on $(0,{p_\ast}]$ differs at least by two. On the other hand, by the zero number \cite{MPNu13}, the number of sign changes of $\Psi^\mathsf{T}(t)$ and $\dot {x}^\ast(t)$ over a unit-length interval may differ by one at most by one. Hence, if ${p_\ast} < 2$, then the pigeonhole principle yields a contradiction.

    To see the case ${p_\ast} \in {J}_1$, we claim that if $q^\mathsf{T} \in [0, {p_\ast})$ is such that $\Psi^\mathsf{T}(q^\mathsf{T}) = 0$, then $\dot {x}^\ast(t)$ changes signs in the interval $(q^\mathsf{T} - 1, q^\mathsf{T})$. By contradiction, suppose that $\dot {x}^\ast(t) \neq 0$ for all $t \in (q^\mathsf{T} - 1, q^\mathsf{T})$. By the zero number \cite{MPNu13}, we have that $\Psi^\mathsf{T}(t)$ has a constant sign on $(q^\mathsf{T} - 1, q^\mathsf{T})$. However, from Proposition \ref{propsition-derivative-bilinear-form}, we have that
    \begin{align}
        \begin{split}
            \left[\Psi^\mathsf{T}_{q^\mathsf{T}}, \dot {x}^\ast_{q^\mathsf{T}}\right]_{q^\mathsf{T}} &= \int_{q^\mathsf{T}}^{{q^\mathsf{T}} + 1} \Psi^\mathsf{T}(t)B(t)\dot {x}^\ast(t)\, \mathrm{d} t \\
            &= \left[\Psi^\mathsf{T}_{0}, \dot {x}^\ast_{0}\right]_{0} \\
            &= 0,
        \end{split}
    \end{align}
    and $\dot {x}^\ast(t)$ changes signs on $(q^\mathsf{T} - 1, q^\mathsf{T})$. Thus, the zeros of $\Psi^\mathsf{T}(t)$ and $\dot {x}^\ast(t)$ are in bijection over $[0, {p_\ast})$ and the proof is complete.
\end{proof}

\begin{proof}[Proof of Lemma \ref{lemma-projections}]
    Naturally, we consider the projections constructed in Lemma \ref{lemma-adjoint-eigenfunctions}. All we need to show is that $P_{\Psi}\varrho_{p_\ast} \neq 0$. Recall that $\varrho(t) = \partial_3 S(t, x_0^\ast; r_\ast)$ solves the initial value problem \eqref{IVP-rho}. Therefore, all we have to show is
    \begin{align}
        \int_0^{p_\ast} \Psi^\mathsf{T}(s) \dot {x}^\ast(s) \, \mathrm{d} s \neq 0
    \end{align}
    For clarity, we work under the additional assumption that $A(t) \equiv 0$ in \eqref{IVP-rho}. \textit{Step 4} below discusses the general scenario. We proceed by contradiction, in the following we suppose that
    \textbf{\begin{align}\label{lemma-nonzero-projection-eq2}
       \int_0^{p_\ast} \Psi^\mathsf{T}(s) \dot {x}^\ast(s) \, \mathrm{d} s = 0.
    \end{align}}\\
    \textit{Step 1:} First, we shall prove that
    \begin{align}\label{step-2-chi-integral}
        \begin{split}
            \int_0^{{p_\ast}} \Psi^\mathsf{T}(s) \ddot {x}^\ast(s) \, \mathrm{d} s = 0.
        \end{split}
    \end{align}
    Notice that $\eta(t) := t \dot {x}^\ast(t)$ solves the forced DDE 
    \begin{align}
        \dot \eta(t) = B(t)\eta(t - 1) + \dot {x}^\ast(t) + \ddot {x}^\ast(t),
    \end{align}
    and satisfies $\eta_{p_\ast} = \lambda\eta_0 + {p_\ast}\dot {x}^\ast_0$. Therefore, we obtain that
    \begin{align}\label{psi-relation-1}
        \begin{split}
            \left[\Psi^\mathsf{T}_{p_\ast}, \eta_{p_\ast}\right]_{p_\ast} &= \left[\Psi^\mathsf{T}_0, \eta_0\right]_0 + \int_0^{p_\ast}\Psi^\mathsf{T}(t)\dot {x}^\ast(t)\, \mathrm{d} t + \int_0^{p_\ast}\Psi^\mathsf{T}(t)\ddot {x}^\ast(t)\, \mathrm{d} t 
        \end{split}
    \end{align}
    and, using the identities \eqref{equation-formal-adjoint-eigenfunctions}, we obtain
    \begin{align}\label{psi-relation-2}
        \begin{split}
            \left[\Psi^\mathsf{T}_{p_\ast}, \eta_{p_\ast}\right]_{p_\ast} &= \left[\Psi^\mathsf{T}_{p_\ast}, \eta_0\right]_{p_\ast} + {p_\ast} \left[\Psi^\mathsf{T}_{p_\ast}, \dot {x}^\ast_0\right]_{p_\ast} \\
            &= \left[\Psi^\mathsf{T}_0, \eta_0\right]_0 + {p_\ast} \left[\Psi^\mathsf{T}_0, \dot {x}^\ast_0\right]_0 \\
            &= \left[\Psi^\mathsf{T}_0, \eta_0\right]_0.
        \end{split}
    \end{align}
    Combining the relations \eqref{lemma-nonzero-projection-eq2}, \eqref{psi-relation-1}, and \eqref{psi-relation-2}, we conclude that
    \begin{align}
        \int_0^{p_\ast}\Psi^\mathsf{T}(t)\ddot {x}^\ast(t)\, \mathrm{d} t = -\int_0^{p_\ast}\Psi^\mathsf{T}(t)\dot {x}^\ast(t)\, \mathrm{d} t = 0.
    \end{align}\\
    \textit{Step 2:} If we denote the time of depth by $q$, then the following identities hold:
    \begin{align}\label{step-3-zero-integrals}
        \int_0^1 \Psi^\mathsf{T}(t)\ddot {x}^\ast(t)\, \mathrm{d} t &= \int_1^{p_\ast} \Psi^\mathsf{T}(t)\ddot {x}^\ast(t)\, \mathrm{d} t = 0 \quad  \text{ and}\\\label{step-3-zero-integrals-q}
        \int_{q}^{q + 1} \Psi^\mathsf{T}(t)\ddot {x}^\ast(t)\, \mathrm{d} t &= \int_{q}^{q + {p_\ast}} \Psi^\mathsf{T}(t)\ddot {x}^\ast(t)\, \mathrm{d} t = 0.
    \end{align}
    Indeed, let us define the auxiliary function $\chi(t)$ that solves
    \begin{align}
        \begin{split}\label{equation-eta}
                \dot \chi(t) &= B(t)\chi(t - 1) + \ddot {x}^\ast(t),\\
            \chi_0(\vartheta) &= 0,\quad \vartheta\in[-1, 0].
        \end{split}
    \end{align}
    Direct integration of \eqref{equation-eta} shows that $\chi(t)$ is given by
    \begin{align}\label{eta-explicit}
        \chi(t) = \begin{cases}
            0,& t\in[-1,0],\\
            \dot {x}^\ast(t),&t\in [0,1].
        \end{cases}
    \end{align}
    On the one hand, by Proposition \ref{propsition-derivative-bilinear-form}, we have that
    \begin{align}
        \begin{split}
            \left[\Psi_1^\mathsf{T}, \chi_1\right]_1 &= \left[\Psi_0^\mathsf{T}, \chi_0\right]_0 + \int_0^1 \Psi^\mathsf{T}(t) \ddot {x}^\ast(t) \, \mathrm{d} t\\
            &= \int_0^1 \Psi^\mathsf{T}(t) \ddot {x}^\ast(t) \, \mathrm{d} t,
        \end{split}
    \end{align}
    on the other hand, by the identities \eqref{equation-formal-adjoint-eigenfunctions} and \eqref{eta-explicit}, we obtain
    \begin{align}
        \begin{split}
            \left[\Psi_1^\mathsf{T}, \chi_1\right]_1 &= \left[\Psi_1^\mathsf{T}, \dot {x}^\ast_1\right]_1 \\
            &= 0.            
        \end{split}
    \end{align}
    Hence, we just showed 
    \begin{equation}
        \int_0^1 \Psi^\mathsf{T}(t) \ddot {x}^\ast(t) \, \mathrm{d} t = 0,
    \end{equation}
    and $\int_1^{p_\ast} \Psi^\mathsf{T}(t) \ddot {x}^\ast(t) \, \mathrm{d} t = 0$ follows by \textit{Step 1}. 
    
    To see \eqref{step-3-zero-integrals-q}, notice that, since $\dot {x}^\ast(q) = 0$, the argument can be replicated if we consider the $\chi$-equation \eqref{equation-eta} with initial time $q$, thus showing
    \begin{align}
        \int_q^{q + 1}\Psi^\mathsf{T}(t) \ddot {x}^\ast(t) \, \mathrm{d} t = 0.
    \end{align}
    Finally, since $\gamma_{\ast}$ is nonhyperbolic, Proposition \ref{proposition-spectrum} yields the critical eigenvalue $\mu_\mathrm{c} = 1$ and, by the periodicity of $\Psi^\mathsf{T}(t)$ in \eqref{pure-eigenfunction}, we obtain
    \begin{align}
        \int_{s}^{s + {p_\ast}}\Psi^\mathsf{T}(t) \ddot {x}^\ast(t) \, \mathrm{d} t = 0, \quad \text{for all } s\in \mathbb{R},
    \end{align}
    this shows the identities \eqref{step-3-zero-integrals-q}.

    \textit{Step 3 ($n$ odd):} Let ${p_\ast} \in {J}_n$ with $n$ odd, then Lemma \ref{lemma-relative-zero-position} shows that ${n_1} := \lfloor(n - 1)/2 \rfloor$ is the unique ${n_1} \in \mathbb{N}_0$ such that ${n_1} {p_\ast} < 1 < ({n_1} + 1) {p_\ast}$ and
    \begin{align}
        {n_1} {p_\ast} < 1 < q + {n_1} {p_\ast} < q + 1 < ({n_1} + 1){p_\ast}.
    \end{align}    
    By \textit{Step 2}, we have that
    \begin{align}\label{zero-integrals-small-domain}
        \int_{{n_1} {p_\ast}}^1 \Psi^\mathsf{T}(t)\ddot {x}^\ast(t) \, \mathrm{d}  t = 0 \quad \text{and} \quad \int_{q+{n_1} {p_\ast}}^{q + 1} \Psi^\mathsf{T}(t)\ddot {x}^\ast(t) \, \mathrm{d}  t = 0.
    \end{align}
    For $t\in({n_1} {p_\ast}, ({n_1} + 1){p_\ast})$, the only zeros of $\ddot {x}^\ast(t)$ lie at $1$ and $q + 1$. Hence, $\ddot {x}^\ast(t)$ has a definite sign over $({n_1} {p_\ast}, 1)$ and $(q + {n_1} {p_\ast}, q + 1)$. Combining the identities \eqref{zero-integrals-small-domain} and Lemma \ref{lemma-adjoint-eigenfunctions}, we conclude that the only zeros of $\Psi^\mathsf{T}(t)$ over the interval $({n_1} {p_\ast}, ({n_1} +1){p_\ast})$ lie at $s_1 \in ({n_1} {p_\ast}, 1)$ and $s_2 \in (q + {n_1} {p_\ast}, q + 1)$.

    Next, consider the function 
    \begin{align}\label{function-integral}
        {X}(t) := \int_{0}^t \Psi^\mathsf{T}(s) \ddot {x}^\ast(s) \, \mathrm{d} s, \quad t\in ({n_1} {p_\ast}, ({n_1} + 1){p_\ast}),
    \end{align}
    by the arguments above, the local extrema of ${X}(t)$ occur at times $s_1 < 1 < s_2 < q + 1$ satisfying
    \begin{align}\label{extrema-position-odd}
        {n_1} {p_\ast} < s_1 < 1 < q + {n_1} {p_\ast} < s_2 < q + 1 < ({n_1} + 1){p_\ast},
    \end{align}
    From \textit{Step 2}, we have that ${X}(1) = {X}({p_\ast}) = 0$. Together with the ordering of the extrema in \eqref{extrema-position-odd}, we obtain that $X(q  + {n_1} {p_\ast}) \neq  X(q + 1)$. However, this is in contradiction to \eqref{zero-integrals-small-domain}; see Figure \ref{fig7}.
    
    \textit{Step 3 ($n$ even):} The argument is analogous to the odd case, but the relative placement of the zeros of $\Psi^\mathsf{T}(t)$ changes. Let ${p_\ast} \in {J}_n$ with $n$ even, then   Lemma \ref{lemma-relative-zero-position} yields a unique ${n_1} \in \mathbb{N}_0$ such that ${n_1} {p_\ast} < 1 < ({n_1} + 1) {p_\ast}$ and
    \begin{align}
        {n_1} {p_\ast} < q + {n_1} {p_\ast} + 1 - {p_\ast} < q + {n_1} {p_\ast} < 1 < ({n_1} + 1){p_\ast}.
    \end{align}    
    Notice that $\ddot {x}^\ast(t)$ has constant sign over $(q + 1 - {p_\ast}, q + {n_1} {p_\ast})$ and $(1, ({n_1} + 1){p_\ast})$ and notice that \textit{Step 2} yields the identities 
    \begin{align}\label{zero-integrals-small-domain-b}
        \int_{q + 1 - {p_\ast}}^{q + {n_1} {p_\ast}} \Psi^\mathsf{T}(t)\ddot {x}^\ast(t) \, \mathrm{d}  t = 0 \quad \text{and} \quad \int_{1}^{({n_1} + 1) {p_\ast}} \Psi^\mathsf{T}(t)\ddot {x}^\ast(t) \, \mathrm{d}  t = 0.
    \end{align}
    Thus, the two zeros of $\Psi^\mathsf{T}(t)$ with $t\in ({n_1} {p_\ast}, ({n_1} + 1){p_\ast})$ lie at $s_1$ and $s_2$ such that
    \begin{align}\label{extrema-position-even}
        {n_1} {p_\ast} < q + 1 - {p_\ast} < s_1 < q + {n_1} {p_\ast} < 1 < s_2 < ({n_1} + 1){p_\ast}.
    \end{align}
    In particular, the function ${X}(t)$ defined in \eqref{function-integral} possesses four local extrema at $q + 1 - {p_\ast} < s_1 < q < s_2$. Again, recalling from \textit{Step 1} and \textit{Step 2} that ${X}(1) = {X}({p_\ast}) = 0$, we obtain ${X}(q + 1 - {p_\ast}) \neq {X}(q + {n_1} {p_\ast})$, in contradiction to \eqref{zero-integrals-small-domain-b}; see Figure \ref{fig7}.

    \begin{figure}
        \centering          \includegraphics[width= 0.87\textwidth]{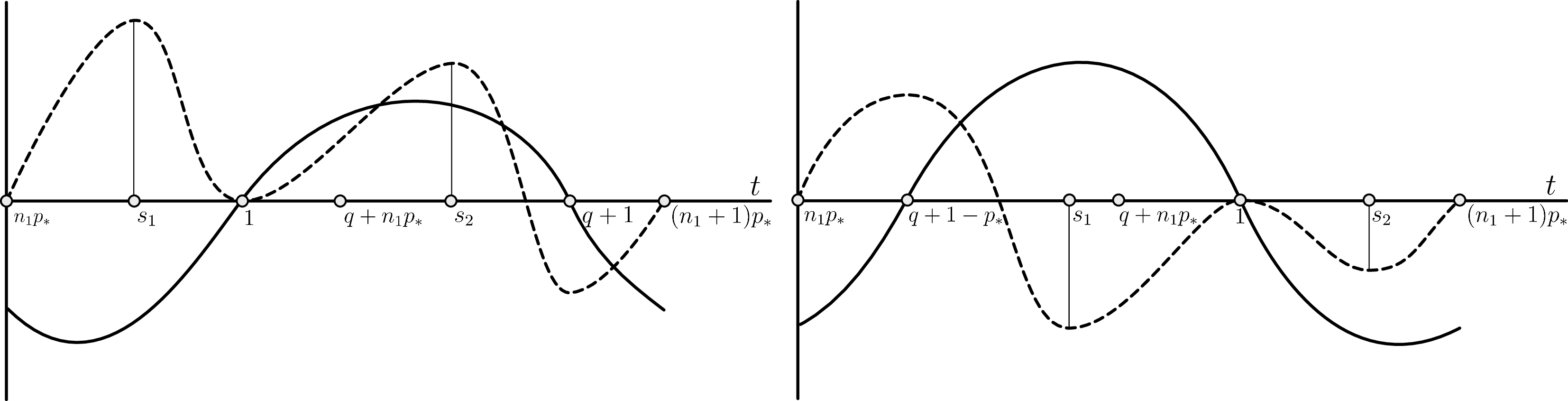}
        \caption{\em Plots of $\ddot {x}^\ast(t)$ (solid) and ${X}(t)$ (dashed) over the interval $({n_1} {p_\ast}, ({n_1} + 1){p_\ast})$.  Left: If $p_\ast \in J_n $ with an odd $n$, then the distribution of extrema of $X(t)$ is given by \eqref{extrema-position-odd}. Together with $X(1) = X(p_\ast) = 0$, we see that $X(q + n_1p_\ast) \neq X(q + 1)$, yielding a contradiction to \eqref{zero-integrals-small-domain}. Right: If $p_\ast \in J_n $ with $n$ even, then the extrema of $X(t)$ satisfy \eqref{extrema-position-even}. As in the odd case, we obtain $X(q + 1 - p_\ast) \neq X(q + n_1p_\ast)$, in contradiction to the integral identity \eqref{zero-integrals-small-domain-b}.}
        \label{fig7}
    \end{figure}
    
    \textit{Step 4:} Finally, we show how to modify the proof for the general case $A(t) \neq 0$. Notice that we can transform $\dot {x}^\ast_0$
    \begin{align}\label{equation-remove-A}
         \xi(t) := \exp\left(-\int_0^t A(s) \, \mathrm{d} s\right)\dot {x}^\ast(t)
    \end{align}
    Then $\xi(t)$ solves
    \begin{align}
        \dot y(t) = \tilde{B}(t) y(t-1),
    \end{align}
    with the modified ${p_\ast}$-periodic coefficient
    \begin{align}
        \tilde{B}(t):= \exp\left(\int_{t-1}^t A(s)\, \mathrm{d} s\right) B(t).
    \end{align}
    In doing so, we have multiplied the spectrum of the monodromy operator $L$ in Proposition \ref{proposition-spectrum} by a positive constant
    \begin{align}
        \exp\left(-\int_0^{p_\ast} A(s)\, \mathrm{d}  s\right) > 0.
    \end{align}
    However, our transformation preserves the sign changes of the eigenfunctions. We perform an analogous transformation so that the formal adjoint equation \eqref{formal-adjoint-equation} becomes
    \begin{align}\label{adjoint-without-A}
        \dot y^\mathsf{T}(t) = - \tilde{B}(t + 1) y^\mathsf{T}(t + 1),
    \end{align}
    with the dual eigenfunction
    \begin{align}\label{eq-new-psi}
        \tilde{\Psi}^\mathsf{T}(t) := \exp\left(\int_0^t A(s)\, \mathrm{d} s\right)\Psi^\mathsf{T}(t),
    \end{align}
    Hence, we have to show
    \begin{align}
        \int_0^{p_\ast}\tilde{\Psi}^\mathsf{T}(t)\xi(t) \, \mathrm{d} t \neq 0.
    \end{align}
    However, the integral identities \eqref{step-2-chi-integral} and \eqref{step-3-zero-integrals}--\eqref{step-3-zero-integrals-q} hold if we replace $\Psi^\mathsf{T}(t)$ by $\tilde{\Psi}^\mathsf{T}(t)$ and $\ddot {x}^\ast(t)$ by $\dot \xi(t)$. Moreover, our transformations \eqref{equation-remove-A} and \eqref{eq-new-psi} preserve the information on the placement of the zeros of $\Psi^\mathsf{T}(t)$ and $\ddot {x}^\ast(t)$. Therefore, arguing as in \textit{Step 3} completes the proof.
\end{proof}

\textbf{Acknowledgement.} The author is greatly indebted to Chun-Hsiung Hsia and Jia-Yuan Dai for their support and for many insightful discussions.

\textbf{Funding.} A. L.-N. has been supported by NSTC grant 113-2123-M-002-009.

\bibliographystyle{abbrv}
\bibliography{references}

\end{document}